\documentclass[12pt,reqno]{amsart}
\usepackage{amsmath, amsthm, amssymb, mathrsfs}

\advance \topmargin by -\headheight
\advance \topmargin by -\headsep
     
\evensidemargin \oddsidemargin
\newtheorem{theorem}{Theorem}[section]
\newtheorem*{theorem*}{Theorem}
\newtheorem{lemma}[theorem]{Lemma}

\newtheorem{conjecture}[theorem]{Conjecture}

\theoremstyle{definition}

\theoremstyle{remark}

\numberwithin{equation}{section}

\def\bfa{{\mathbf a}}
\def\bfb{{\mathbf b}}
\def\bfc{{\mathbf c}}
\def\bfd{{\mathbf d}}

\def\bfm{{\mathbf m}}

\def\bfx{{\mathbf x}}
\def\bfy{{\mathbf y}}
\def\bfz{{\mathbf z}}

\def\calA{\mathscr{A}}   
\def\calB{\mathscr{B}} 
\def\calC{\mathscr{C}}

\def\calE{{\mathcal E}}

\def\calL{\mathscr{L}}

\def\atil{\widetilde a} \def\btil{\widetilde b} \def\ctil{\widetilde c}
\def\dtil{\widetilde{d}} 
\def\Ftil{\widetilde F} \def\Gtil{\widetilde G} \def\Htil{\widetilde H}
\def\ltil{\widetilde l} \def\mtil{\widetilde m}
\def\ntil{\widetilde n} \def\Ntil{\widetilde N}

\def\dbC{{\mathbb C}}\def\dbF{{\mathbb F}}\def\dbN{{\mathbb N}}
\def\dbR{{\mathbb R}}
\def\dbZ{{\mathbb Z}}

\def\grB{{\mathfrak B}}
\def\grC{{\mathfrak C}}
\def\grD{{\mathfrak D}}

\def\grg{\mathfrak g} 

\def\grJ{{\mathfrak J}}

\def\grm{{\mathfrak m}}\def\grM{{\mathfrak M}}\def\grN{{\mathfrak N}}
\def\grn{{\mathfrak n}}\def\grS{{\mathfrak S}}
\def\grB{{\mathfrak B}}\def\grC{{\mathfrak C}}
\def\grk{{\mathfrak k}}\def\grK{{\mathfrak K}}

 \def\grX{{\mathfrak X}}

\def\alp{{\alpha}} 
\def\bet{{\beta}}  
\def\gam{{\gamma}} \def\Gam{{\Gamma}}
\def\del{{\delta}} \def\Del{{\Delta}}
\def\zet{{\zeta}} \def\bfzet{{\boldsymbol \zeta}} \def\Zet{{\rm Z}} \def\Zettil{{\widetilde \Zet}}
\def\bfeta{{\boldsymbol \eta}}
\def\tet{{\theta}}  

\def\lam{{\lambda}}   \def\bflam{{\boldsymbol \lambda}}

\def\bfxi{{\boldsymbol \xi}} \def\Xitil{\widetilde \Xi}

\def\Ups{{\Upsilon}} 
\def\ome{{\omega}} \def\Ome{{\Omega}}
\def\d{{\partial}}
\def\eps{\varepsilon}

\def\le{\leqslant} \def\ge{\geqslant}

\def\d{{\,{\rm d}}}

\begin{document}
\title[Subconvexity for additive equations]{Subconvexity for additive equations:\\ pairs of undenary cubic forms}
\author[J\"org Br\"udern]{J\"org Br\"udern}
\address{JB: Mathematisches Institut, Bunsenstrasse 3--5, D-37073 G\"ottingen, Germany}
\email{bruedern@uni-math.gwdg.de}
\author[Trevor D. Wooley]{Trevor D. Wooley$^*$}
\address{TDW: School of Mathematics, University of Bristol, University Walk, Clifton, Bristol BS8 1TW, United Kingdom}
\email{matdw@bristol.ac.uk}
\thanks{$^*$Supported by a Royal Society Wolfson Research Merit Award.}
\date{}
\begin{abstract} We investigate pairs of diagonal cubic equations with integral coefficients. For a class of such Diophantine systems with $11$ or more variables, we are able to establish that the number of integral solutions in a large box is at least as large as the expected order of magnitude.\end{abstract}
\maketitle

\section{Introduction} The convexity barrier in the Hardy-Littlewood method presents an apparently insurmountable obstacle to the analysis of Diophantine systems in which the underlying number of variables is smaller than twice the total degree of the system. As is well-known, this obstruction arises from the relative sizes of the product of local densities associated with the system, and the square-root of the available reservoir of variables that is a limiting feature of associated exponential sum estimates. In this paper, we establish a lower bound of the anticipated magnitude for the number of integral zeros of certain pairs of diagonal cubic forms in $11$ variables, thereby breaking this convexity barrier.\par

In order to introduce the Diophantine systems central to our discussion, take $l$, $m$, $n$ to be non-negative integers with $m\ge n$, and fix non-zero integers $a_i$, $b_i$, $c_j$, $d_k$, where $1\le i\le l$, $1\le j\le m$ and $1\le k\le n$. We then define $N(B)$ to be the number of integral solutions to the system
\begin{equation}\label{1.1}
\left.\begin{aligned}
&a_1x_1^3+\ldots +a_lx_l^3+c_1y_1^3+\ldots +c_my_m^3&&=0,\\
&b_1x_1^3+\ldots +b_lx_l^3&+d_1z_1^3+\ldots +d_nz_n^3&=0,
\end{aligned}\, \right\}
\end{equation}
with $x_i,y_j,z_k\in [-B,B]$. Associated to the system (\ref{1.1}) are the total number of variables $s=l+m+n$, and a measure of the minimal number of variables across equations $q_0^*=\min\{s-l,s-m,s-n\}$. Before announcing the principal conclusion of this paper, we direct the reader to \S6 for a description of Hooley's Riemann Hypothesis (which we call HRH).

\begin{theorem}\label{theorem1.1} Let $s\ge 11$ and $q_0^*\ge 7$, and suppose that the system $(\ref{1.1})$ admits non-singular $p$-adic solutions for each prime $p$. Then, with the possible exception of the case $(l,m,n)=(5,5,2)$, one has $N(B)\gg B^{s-6}$. In the latter exceptional case one recovers the same conclusion by appealing to HRH. 
\end{theorem}

Examples allied to the system
\begin{equation*}
\left.\begin{aligned}
&49(x_1^3+2x_2^3+3x_3^3)+y_1^3+2y_2^3+7y_3^3+14y_4^3&&=0,\\
&49(x_1^3+4x_2^3+4x_3^3)&+z_1^3+2z_2^3+7z_3^3+14z_4^3&=0,
\end{aligned}\, \right\}
\end{equation*}
demonstrate that the $p$-adic solubility hypothesis in the theorem is required, as the reader may easily verify.\par

The conclusion of Theorem \ref{theorem1.1} establishes the Hasse Principle for those systems (\ref{1.1}) with $s\ge 11$ and $q_0^*\ge 7$, and indeed an appropriate modification of our methods would confirm the weak approximation property for the same systems. In view of the convexity barrier, the cases in which $s=11$ are of particular interest. We note that when $s\ge 13$, the conclusion of Theorem \ref{1.1} follows from our previous work \cite{BW2007a, BW2007b} concerning pairs of diagonal cubic forms (in particular, see \cite[Theorem 2]{BW2007b}). When $s\ge 11$, moreover, the special case in which $\bfa$ and $\bfb$ are in rational ratio is covered by \cite[Theorem 10]{BW2007b}. Previous results on pairs of diagonal cubic equations, meanwhile, apply only to systems having $14$ or more variables (see, in chronological order, the references \cite{DL1966, Coo1972, Vau1977, BB1988, Bru1990}).\par

When $s\ge 12$, which is the threshold of the convexity barrier, we are able to refine the asymptotic lower bound of Theorem \ref{theorem1.1}. In this context, it is useful to introduce the product of local densities associated with the system (\ref{1.1}). The latter we define by $\calC=v_\infty \prod_p v_p$, in which $v_\infty$ is the area of the manifold defined by (\ref{1.1}) in the box $[-1, 1]^s$, and for each prime number $p$ one takes
$$v_p=\lim_{h\rightarrow \infty}p^{h(2-s)}M(p^h),$$
where $M(q)$ denotes the number of solutions of (\ref{1.1}) with $\bfx\in (\dbZ/q\dbZ)^s$.

\begin{theorem}\label{theorem1.2} When $s\ge 12$ and $q_0^*\ge 8$, one has $N(B)\ge (\calC+o(1))B^{s-6}$. 
\end{theorem}

For comparison, our earlier work \cite{BW2011} establishes a conclusion which implies Theorem \ref{theorem1.2} when $s\ge 14$ (see \cite[Theorem 1.1]{BW2011}).\par

Thus far we have discussed only problems involving simultaneous cubic equations, but for problems of very low degree alternative approaches may be applicable. Thus, for systems of linear equations, one has recent work of Green and Tao \cite{GT2010} for prime numbers, and work of the first author \cite{Bru2009} for limit periodic sequences. It is worth remarking also that the Kloosterman method provides conclusions on the edge of subconvexity for problems of quadratic type (for some of the relevant literature, see \cite{Est1962, HB1996, Klo1927}). Problems within the orbit of our methods are not limited to diagonal cubic examples alone, and in \S11 we outline some of what may be said concerning problems of higher degree.\par

In this paper we employ the well-known symbols of Landau and Vinogradov. The constants implicit in the use of these symbols depend at most on $s$, $\bfa$, $\bfb$, $\bfc$, $\bfd$ and $\eps$, unless otherwise indicated. In an effort to simplify our analysis, we adopt the following convention concerning the number $\eps$. Whenever $\eps$ appears in a statement, either implicitly or explicitly, we assert that the statement holds for each $\eps>0$. Note that the ``value'' of $\eps$ may consequently change from statement to statement. Throughout, we take $B$ to be a positive real number with $B\gg 1$, in the sense indicated.

\section{Preliminary considerations} We initiate our discussion by introducing the notation and technical infrastructure necessary for our application of the circle method. Consider a system of the shape (\ref{1.1}) subject to the hypotheses of the statement of Theorem \ref{theorem1.1}. Since the conclusion of this theorem is already supplied by \cite[Theorem 2]{BW2007b}\footnote{We correct an oversight here in the statement of \cite[Theorem 9]{BW2007b}, an ingredient in the proof of \cite[Theorem 2]{BW2007b} relevant to our discussion when the system (\ref{1.1}) has the shape $(1,m,n)$ with $m\ge n\ge 6$. The statement of the former theorem should read as follows.
\begin{theorem*} Suppose that $t$ is a natural number with $t\ge 6$, and let $c_1,\ldots ,c_t$ be natural numbers satisfying $(c_1,\ldots ,c_t)=1$. Then for each natural number $d$ there is a positive number $\Del$, depending at most on $\bfc$ and $d$, with the property that the set $\calE_t(P)$, defined by
$$\calE_t(P)=\{n\in \dbN : \text{$\nu Pd^{-1/3}<n\le Pd^{-1/3}$, $(n,c_1\ldots c_t)=1$ and $R_t(dn^3;\bfc)<\Del P^{t-3}$}\},$$
has at most $P^{1-\tau}$ elements.
\end{theorem*}
\noindent The inserted condition $(n,c_1\ldots c_t)=1$ should also be imposed in the subsequent application of this theorem in \cite[\S6]{BW2007b}. Our forthcoming work \cite{BW2012} discusses the case $(l,m,n)=(1,6,6)$ as a particular instance of more general investigations of diagonal senary cubic forms.
}
 when $s\ge 13$, there is no loss of generality in restricting to the situations wherein $s=11$ or $s=12$. A modicum of computation reveals that the triple $(l,m,n)$ associated with the system (\ref{1.1}) must take one of four shapes, namely:
\begin{enumerate}
\item[(A)] $(3,4,4)$ or $(3,5,4)$,
\item[(B)] $(4,4,3)$, $(4,4,4)$, $(4,5,3)$ or $(5,4,3)$,
\item[(C)] $(2,5,5)$,
\item[(D)] $(5,5,2)$.
\end{enumerate}

\par Systems of type A and B we analyse by very similar methods in \S\S3 and 4, respectively. The reader will find that the ideas developed in these sections serve as a model for the treatment of the remaining cases relevant to the proof of Theorem \ref{theorem1.1}, as well as the cases required to establish Theorem \ref{theorem1.2}. In order to handle systems of type C, we `borrow' a variable from each of the long blocks of $5$, adding them to the short block of $2$. In this way we obtain a system superficially resembling those of type B, though sharing characteristics with those of type A. In this way, we are able in \S5 to offer an economical treatment of systems of type C that rests heavily on the work of \S\S3 and 4. Readers may care to challenge themselves with the task of developing an alternative treatment based on our work \cite{ARTS1} joint with Kawada, in which the system (\ref{1.1}) is understood in terms of an exceptional set problem involving the representation of values of a binary diagonal form by a diagonal form in five variables. Finally, in order to accommodate systems of type D, we first develop mean value estimates for exponential sums conditional on HRH, and then adapt the methods used for our analysis of systems of type A and B. We offer an abbreviated account of this work in \S6.\par

It is apparent that the system (\ref{1.1}) possesses a real solution $(\bfx,\bfy,\bfz)=(\bfxi,\bfeta,\bfzet)\in (-1,1)^s$ in which $\xi_i$, $\eta_i$ and $\zet_i$ are each positive for $i\ne 1,2$. We put
$$\nu=\tfrac{1}{2}\min_{i\ne 1,2}\{ \xi_i,\eta_i,\zet_i\}.$$
Let $\eta$ be a small positive number to be fixed in due course, and write
$$\calA_\eta (B)=\{ n\in \dbZ\cap [1,B]: \text{$p$ prime and $p|n$}\Rightarrow p\le B^\eta\}.$$
We then put
$$\calA_\eta^*(B)=\{n\in [-B,B]: \text{$|n|\in \calA_\eta(B)$ or $n=0$}\}.$$
Define the exponential sums
$$f(\tet)=\sum_{|x|\le B}e(\tet x^3),\quad g(\tet)=\sum_{\nu B<x\le B}e(\tet x^3),\quad h(\tet)=\sum_{x\in \calA_\eta^*(B)}e(\tet x^3),$$
where, as usual, we write $e(z)$ for $e^{2\pi iz}$. Let $\tau_0$ be the positive number defined via the relation $\tau_0^{-1}=852+16\sqrt{2833}=1703.6\ldots $. Then, when $a$ and $b$ are fixed non-zero integers and $\tau_1$ is any real number with $\tau_1<\tau_0$, the methods of \cite{Woo2000} may be applied to confirm that whenever $\eta$ is a sufficiently small positive number, one has
\begin{equation}\label{2.1}
\int_0^1|g(a\tet)^2h(b\tet)^4|\d\tet \ll B^{13/4-\tau_1}.
\end{equation}
We direct the reader to \cite[\S5]{Woo1995b} and \cite[\S2]{Woo2000} for the necessary ideas, the presence of the coefficients $a$ and $b$ leading to superficial complications only. We put $\tau=\frac{1}{10}\tau_0$, and for the remainder of this paper we fix our choice of $\eta>0$ to be sufficiently small in the context of the upper bound (\ref{2.1}) with $\tau_1=9\tau$.\par 

Having introduced the cast of exponential sums to appear in our application of the circle method, we next introduce the generating functions
\begin{equation}\label{2.2}
F(\alp,\bet)=h(a_1\alp+b_1\bet)h(a_2\alp+b_2\bet)\prod_{i=3}^lg(a_i\alp+b_i\bet),
\end{equation}
\begin{equation}\label{2.3}
G(\alp)=h(c_1\alp)h(c_2\alp)\prod_{j=3}^mg(c_j\alp),\quad H(\bet)=h(d_1\bet)h(d_2\bet)\prod_{k=3}^ng(d_k\bet).
\end{equation}
Here we adopt the convention that an empty product is equal to unity. When $\grB\subseteq [0,1)^2$ is measurable, we define
\begin{equation}\label{2.4}
N(B;\grB)=\iint_\grB F(\alp,\bet)G(\alp)H(\bet)\d\alp\d\bet.
\end{equation}
Then, by orthogonality, one finds that
\begin{equation}\label{2.5}
N(B)\ge N(B;[0,1)^2).
\end{equation}

At this stage we introduce the primary Hardy-Littlewood dissection. We take the major arcs $\grM$ to be the union of the intervals
$$\grM(q,a)=\{\alp\in [0,1):|q\alp-a|\le B^{-9/4}\},$$
with $0\le a\le q\le B^{3/4}$ and $(a,q)=1$. The corresponding set of minor arcs $\grm$ is defined by putting $\grm=[0,1)\setminus \grM$. In addition, we define a two-dimensional Hardy-Littlewood dissection as follows. With an eye towards concision in future sections, we put
$$L=\log B,\quad \calL=\log L\quad \text{and}\quad Q=L^{1/100}.$$
We then define the narrow major arcs $\grN$ to be the union of the boxes
$$\grN(q,a,b)=\{ (\alp,\bet)\in [0,1)^2:\text{$|\alp-a/q|\le QB^{-3}$ and $|\bet-b/q|\le QB^{-3}$}\},$$
with $0\le a,b\le q\le Q$ and $(a,b,q)=1$. The complementary set of minor arcs is $\grn=[0,1)^2\setminus \grN$. Finally, we write
\begin{equation}\label{2.6}
\grK=(\grM\times \grM)\setminus \grN.
\end{equation}
We regard sets of major and minor arcs throughout as subsets of $\dbR/\dbZ$, or the appropriate higher dimensional analogue of the latter. Thus, for example, when we write $\gam\in \grm$, then we are implicitly asserting that $\gam\in \grm+\dbZ$.\par

Our strategy for obtaining a lower bound for $N(B;[0,1)^2)$ employs the Hardy-Littlewood method, of course, though in a somewhat unconventional manner. We begin by analysing the contribution of the narrow set of major arcs $\grN$.

\begin{lemma}\label{lemma2.1} Suppose that the system (\ref{1.1}) admits non-singular $p$-adic solutions for each prime number $p$. Then for systems of type A, B, C and D, one has $N(B;\grN)\gg B^{s-6}$.
\end{lemma}

\begin{proof} We begin by defining the Gauss sum
$$S(q,a)=\sum_{r=1}^qe(ar^3/q),$$
and then introduce the expression
\begin{equation}\label{2.7}
A(q)=\underset{(u,v,q)=1}{\sum_{u=1}^q\sum_{v=1}^q}\,T(q,u,v),
\end{equation}
where
$$T(q,u,v)=\prod_{i=1}^lS(q,a_iu+b_iv)\prod_{j=1}^mS(q,c_ju)\prod_{k=1}^nS(q,d_kv).$$
In addition, we write
$$v(\tet)=\int_{\nu B}^Be(\tet \gam^3)\d\gam \quad \text{and}\quad w(\tet)=\int_{-B}^Be(\tet \gam^3)\d\gam ,$$
and then put
$$V(\xi,\zet)=V_F(\xi,\zet)V_G(\xi)V_H(\zet),$$
where
$$V_F(\xi,\zet)=w(a_1\xi+b_1\zet)w(a_2\xi+b_2\zet)\prod_{i=3}^lv(a_i\xi+b_i\zet),$$
$$V_G(\xi)=w(c_1\xi)w(c_2\xi)\prod_{j=3}^mv(c_j\xi),\quad V_H(\zet)=w(d_1\zet)w(d_2\zet)\prod_{k=3}^nv(d_k\zet).$$
Finally, we define
$$\grJ (X)=\iint_{\calB(X)}V(\xi,\zet)\d\xi\d\zet \quad \text{and}\quad \grS(X)=\sum_{1\le q\le X}A(q),$$
in which we have written $\calB(X)$ for the box $[-XB^{-3},XB^{-3}]^2$. Then by following the argument of \cite[\S7]{BW2007a} leading to \cite[equation (7.8)]{BW2007a}, one finds that there exists a positive constant $C$ with the property that
\begin{equation}\label{2.8}
N(B;\grN)-C\grS(Q)\grJ(Q)\ll B^{s-6}L^{-1/4}.
\end{equation}

\par In order to estimate the truncated singular series $\grS(Q)$, we begin by following the argument of the proof of the estimate \cite[(7.14)]{BW2007a} presented on page 890 of the latter paper. In the present context we find that there is an integer $t\ge 3$, and a non-zero integer $\Del$ depending on $\bfa$, $\bfb$, $\bfc$, $\bfd$, such that
$$A(q)\ll q^{-s/3}\sum_{\substack{v_1,\ldots,v_t\\ v_1\ldots v_t|\Del q}}\frac{q^2}{v_1\ldots v_t}(v_1^{r_1}\ldots v_t^{r_t})^{1/3},$$
in which $r_1,\ldots,r_t$ are positive integers satisfying $r_1+\ldots +r_t=s$, and further $\max_ir_i=4$ when $s=11$, and $\max_ir_i\le 5$ when $s=12$ . Consequently, one has $A(q)\ll q^{\eps-4/3}$. An inspection of the proof of \cite[Lemma 12]{BW2007a}, noting \cite[equation (7.14)]{BW2007a}, reveals that $\grS=\underset{X\rightarrow \infty}{\lim}\grS(X)$ exists, that $\grS-\grS(X)\ll X^{-1/4}$, and thus $\grS(Q)\gg 1$. A pedestrian modification of the proof of \cite[Lemma 13]{BW2007a}, on the other hand, reveals that in the present context one has $\grJ(Q)\gg B^{s-6}$. In combination with the lower bound $\grS(Q)\gg 1$ just obtained, the conclusion of the lemma is evident from the relation (\ref{2.8}).
\end{proof}

A detailed account of the next step in our analysis, a comparison of $N(B;\grN)$ with $N(B;\grM\times \grM)$, depends on the particular case at hand, and this we defer to the following four sections. However, we take the opportunity now to record several auxiliary estimates of significance in the discussion to come. We begin by considering certain major arc integrals.

\begin{lemma}\label{lemma2.2} Let $a$ be a fixed non-zero integer, and let $b$ be a non-zero rational number. Then one has
\begin{equation}\label{2.9}
\sup_{\lam\in \dbR}\int_\grM |g(a\tet)h(b\tet+\lam)|^2\d\tet \ll B^{1+\eps},
\end{equation}
and when $\del>0$,
\begin{equation}\label{2.10}
\sup_{\lam\in \dbR}\int_\grM |g(a\tet)|^{2+\del}|h(b\tet+\lam)|^2\d\tet \ll B^{1+\del}.
\end{equation}
\end{lemma}

\begin{proof} The upper bound (\ref{2.10}) is immediate from \cite[Lemma 9]{BW2007a}. By taking $\del=0$ in the argument of the proof of the latter, meanwhile, one obtains (\ref{2.9}) (see also \cite[Lemma 3.4]{ARTS1}).
\end{proof}

As an immediate consequence of this lemma, we obtain major arc estimates for $G(\alp)$ and $H(\bet)$.

\begin{lemma}\label{lemma2.3} For systems of type A and C, one has
$$\int_\grM |G(\alp)|\d\alp \ll B^{m-3+\eps} \quad \text{and}\quad \int_\grM |H(\bet)|\d\bet \ll B^{n-3+\eps}.$$
The former estimate holds also for systems of type B and D.
\end{lemma}

\begin{proof} For systems of type A, B, C and D, one has $m\ge 4$. Thus, by applying a trivial estimate for $g(\tet)$ in combination with Schwarz's inequality and Lemma \ref{lemma2.2}, one obtains
$$\int_\grM|G(\alp)|\d\alp \le g(0)^{m-4}\prod_{i=1}^2\Bigl( \int_\grM |g(c_{2+i}\alp)h(c_i\alp)|^2\d\alp\Bigr)^{1/2}\ll B^{m-4}(B^{1+\eps}).$$
For systems of type A and C, meanwhile, one has $n\ge 4$. Hence, by following a similar argument to that just described, with $H(\bet)$ in place of $G(\alp)$, the second assertion of the lemma is confirmed in like manner.
\end{proof}

\par We finish by recording two mean value estimates of some generality. In this context, when $r\ge t\ge 3$ and $\lam_1,\ldots,\lam_r$ are fixed non-zero integers, we write
$$E_t(\tet)=h(\lam_1\tet)h(\lam_2\tet)\prod_{i=3}^tg(\lam_i\tet).$$

\begin{lemma}\label{lemma2.4} One has
$$\int_0^1|E_t(\tet)|^2\d\tet \ll_\bflam B^{2t-11/4-9\tau}\quad (3\le t\le r),$$
and
$$\int_\grm |E_t(\tet)|^2\d\tet \ll_\bflam B^{2t-13/4-8\tau}\quad (4\le t\le r).$$
\end{lemma}

\begin{proof} An application of Schwarz's inequality, combined with the trivial estimate $|g(\lam_i\tet)|\le B$, reveals that
$$\int_0^1|E_t(\tet)|^2\d\tet \le B^{2t-6}\prod_{i=1}^2\Bigl( \int_0^1|g(\lam_3\tet)^2h(\lam_i\tet)^4|\d\tet \Bigr)^{1/2}.$$
From here, the first estimate of the lemma follows from (\ref{2.1}). In order to confirm the second estimate, we begin by noting that a modified version of Weyl's inequality (see \cite[Lemma 1]{Vau1986}) supplies the bound
\begin{equation}\label{2.11}
\sup_{\tet\in \grm}|g(\lam_t\tet)|\ll B^{3/4+\eps}.
\end{equation}
Thus, on utilising the mean value estimate just obtained, we deduce that
\begin{align*}\int_\grm |E_t(\tet)|^2\d\tet &\le \Bigl( \sup_{\tet\in\grm}|g(\lam_t\tet)|\Bigr)^2\int_0^1|E_{t-1}(\tet)|^2\d\tet\\
&\ll (B^{3/4+\eps})^2B^{2(t-1)-11/4-9\tau},
\end{align*}
and the second estimate of the lemma follows.
\end{proof}

\section{Systems of type A} The strategy employed in our proof of Theorem \ref{theorem1.1} is somewhat circuitous, and we illustrate ideas in this section by concentrating on systems of type A. We may assume for the present, therefore, that $l=3$ and $m\ge n\ge 4$. Our strategy for obtaining a lower bound for the number of solutions of the system (\ref{1.1}) counted by $N(B)$ involves constraining the first block of $l$ variables. Restricting in such a manner that the corresponding diagonal forms $a_1x_1^3+\ldots +a_lx_l^3$ and $b_1x_1^3+\ldots +b_lx_l^3$ behave essentially as expected so far as multiplicity of representations is concerned, we obtain a modified counting function $N_0(B)$ whose behaviour is mollified by this arithmetic smoothing. In Lemma \ref{lemma3.3} we show that the major arc contribution within $N_0(B)$ is close to the corresponding contribution within $N(B)$. By means of a pruning operation discussed in Lemma \ref{lemma3.1}, this contribution is seen via Lemma \ref{lemma2.1} to have order of growth $B^{s-6}$. Meanwhile, the minor arc contribution within $N_0(B)$ exploits the available smoothing by means of Bessel's inequality, and is described in mixed form in Lemma \ref{lemma3.4}, and pure minor arc form in Lemma \ref{lemma3.5}. In this way, we aim to show that
$$N(B)\ge N_0(B)\ge N(B;\grN)+o(B^{s-6})\gg B^{s-6},$$
and thereby establish Theorem \ref{theorem1.1}.\par

Our first step in the above plan is to prune from the set of arcs $\grM\times \grM$ to the narrow major arcs $\grN$. We apply Lemmata \ref{lemma2.2} and \ref{lemma2.3} in order to estimate the contribution of the set of arcs $\grK$ defined in (\ref{2.6}), a set we divide into the two subsets
$$\grK_0=\{(\alp,\bet)\in \grK:a_l\alp+b_l\bet\in \grM\}$$
and
$$\grK_1=\{ (\alp,\bet)\in \grK:a_l\alp+b_l\bet\in \grm\}.$$

\begin{lemma}\label{lemma3.1} For systems of type A, one has $N(B;\grK)\ll B^{s-6}\calL^{-1}$.
\end{lemma}

\begin{proof} We first consider the contribution of the set $\grK_1$. Observe that as a consequence of the modified version of Weyl's inequality (\ref{2.11}), one has
$$\sup_{(\alp,\bet)\in \grK_1}|F(\alp,\bet)|\ll B^{l-1}\sup_{a_l\alp+b_l\bet\in \grm}|g(a_l\alp+b_l\bet)|\ll B^{l-1/4+\eps}.$$
Then we deduce from Lemma \ref{lemma2.3} that
\begin{align}
N(B;\grK_1)&\ll \Bigl(\sup_{(\alp,\bet)\in \grK_1}|F(\alp,\bet)|\Bigr)\int_\grM |G(\alp)|\d\alp\int_\grM|H(\bet)|\d\bet \notag \\
&\ll B^{l-1/4+\eps}(B^{m-3+\eps})(B^{n-3+\eps})\ll B^{s-49/8}.\label{3.1}
\end{align}

\par Turning our attention next to the contribution from the set $\grK_0$, we begin by considering the functions
\begin{align}
\Phi_G(\alp)&=|g(c_3\alp)g(c_4\alp)|^{5/4}|h(c_1\alp)h(c_2\alp)|,\label{3.2}\\
\Phi_H(\bet)&=|g(d_3\bet)g(d_4\bet)|^{5/4}|h(d_1\bet)h(d_2\bet)|,\notag
\end{align}
and their mean values
$$I_G=\int_\grM \Phi_G(\alp)\d\alp \quad \text{and}\quad I_H=\int_\grM\Phi_H(\bet)\d\bet .$$
With an application of Schwarz's inequality mirroring that employed in the proof of Lemma \ref{lemma2.3}, followed by recourse to Lemma \ref{lemma2.2}, one obtains
\begin{equation}\label{3.3}
I_G\le \prod_{i=1}^2\Bigl( \int_\grM |g(c_{2+i}\alp)|^{5/2}|h(c_i\alp)|^2\d\alp\Bigr)^{1/2}\ll B^{3/2},
\end{equation}
and a symmetric argument yields $I_H\ll B^{3/2}$. Writing
\begin{equation}\label{3.4}
J_0=\int_\grM\int_\grM\Phi_G(\alp)\Phi_H(\bet)\d\alp\d\bet ,
\end{equation}
therefore, we deduce that $J_0=I_GI_H\ll B^3$.\par

Next, when $i\in \{1,2\}$, $k\in \{l-1,l\}$ and $E$ is either $G$ or $H$, define
\begin{equation}\label{3.5}
J_{i,k}^E=\iint_{\grK_0}|g(a_k\alp+b_k\bet)|^{5/2}|\Phi_E(\alp)h(a_i\alp+b_i\bet)^2|\d\alp\d\bet .
\end{equation}
We note that at present we require these integrals only when $k=l$, though in \S4 we make use of them also when $k=l-1$. By means of a change of variable, one discerns from (\ref{3.3}) and Lemma \ref{lemma2.2} that
\begin{align}
J_{i,l}^G&\ll \int_\grM \Phi_G(\alp)\sup_{\lam\in \dbR}\int_\grM|g(\gam)|^{5/2}|h(b_ib_l^{-1}\gam+\lam)|^2\d\gam\d\alp \notag \\
&\ll B^{3/2}I_G\ll B^3,\label{3.6}
\end{align}
and in an analogous manner one obtains $J_{i,l}^H\ll B^3$. Finally, we put
$$\Psi(\alp,\bet)=\prod_{i=1}^2|h(c_i\alp)h(d_i\bet)h(a_i\alp+b_i\bet)^3|.$$
Then, as a consequence of the argument of the proof of \cite[Lemma 10]{BW2007a} (see in particular the display preceding \cite[equation (6.14)]{BW2007a}), one has
\begin{equation}\label{3.7}
\sup_{(\alp,\bet)\in \grn}\Psi(\alp,\bet)\ll B^{10}Q^{-1/10}.
\end{equation}

\par By applying H\"older's inequality in combination with the estimates assembled above, and applying a trivial bound for $g(\tet)$, we conclude that
\begin{align*}
N(B;\grK_0)&\le g(0)^{s-11}\Bigl(\sup_{(\alp,\bet)\in \grn}\Psi(\alp,\bet)\Bigr)^{1/5}(J_{1,l}^GJ_{2,l}^GJ_{1,l}^HJ_{2,l}^H)^{1/10}J_0^{3/5}\\
&\ll B^{s-11}(B^{10}Q^{-1/10})^{1/5}(B^{12})^{1/10}(B^3)^{3/5}=B^{s-6}Q^{-1/50}.\end{align*}
On recalling (\ref{3.1}), therefore, we obtain the upper bound
$$N(B;\grK)=N(B;\grK_0)+N(B;\grK_1)\ll B^{s-6}Q^{-1/50},$$
and this completes the proof of the lemma.
\end{proof}

Thus far our argument presents the appearance of a conventional application of the Hardy-Littlewood method. It is at this point that unconventional elements are introduced. When $u,v\in \dbZ$, denote by $\rho(u,v)$ the number of integral solutions of the system
\begin{align}
a_1y_1^3+\ldots +a_ly_l^3&=u,\label{3.8}\\
b_1y_1^3+\ldots +b_ly_l^3&=v,\label{3.9}
\end{align}
with $y_1,y_2\in \calA_\eta^*(B)$ and $\nu B<y_i\le B$ $(3\le i\le l)$. In addition, write $\rho_1(u)$ for the number of integral solutions of (\ref{3.8}), and $\rho_2(v)$ for the number of integral solutions of (\ref{3.9}), subject to the same conditions on $\bfy$. Then if we put
$$\Ome=\sum_{i=1}^l(|a_i|+|b_i|)\quad \text{and}\quad \grX=[-\Ome B^3,\Ome B^3]\cap \dbZ,$$
one finds that
\begin{equation}\label{3.10}
\rho_1(u)=\sum_{v\in \grX}\rho(u,v)\quad \text{and}\quad \rho_2(v)=\sum_{u\in \grX}\rho(u,v). 
\end{equation}
The arithmetic smoothing to which we alluded in the introduction of this section is achieved by dividing the set $\grX^2$ into three subsets, calibrated by a truncation parameter $T$. We define the sets $\grX_i=\grX_i(T)$ for $i=0,1,2$ by taking
\begin{align}
\grX_0(T)&=\{ (u,v)\in \grX^2:\text{$\rho_1(u)\le T$ and $\rho_2(v)\le T$}\},\notag \\
\grX_1(T)&=\{ (u,v)\in \grX^2:\text{$\rho_1(u)>T$ and $\rho_2(v)\le T$}\},\label{3.11}\\
\grX_2(T)&=\{ (u,v)\in \grX^2:\rho_2(v)>T\},\notag 
\end{align}
so that
\begin{equation}\label{3.12}
\grX_0(T)=\grX^2\setminus (\grX_1(T)\cup \grX_2(T)).
\end{equation}
For systems of type A, we fix the truncation parameter to be $T=B^{l-11/4}$.\par

At this point we pause to establish an auxiliary estimate for the quantity
$$\Xi_i=\sum_{(u,v)\in \grX_i}\rho(u,v)\quad (i=1,2).$$

\begin{lemma}\label{lemma3.2}
For systems of type A, one has $\Xi_i\ll B^{l-9\tau}$ $(i=1,2)$.
\end{lemma}

\begin{proof} Observe first that in view of (\ref{3.10}) we have
$$\Xi_1\le \sum_{\substack{u\in \grX\\ \rho_1(u)>B^{l-11/4}}}\sum_{v\in \grX}\rho(u,v)\le B^{11/4-l}\sum_{u\in \grX}\rho_1(u)^2.$$
On considering the underlying Diophantine equation and applying Lemma \ref{2.4}, one sees that
\begin{equation}\label{3.13}
\sum_{u\in \grX}\rho_1(u)^2=\int_0^1|F(\alp,0)|^2\d\alp\ll B^{2l-11/4-9\tau}.
\end{equation}
We therefore conclude that when $i=1$, one has $\Xi_i\ll B^{l-9\tau}$, and a symmetrical variant of this argument delivers the same bound when $i=2$. 
\end{proof}

Next, when $\grC$ and $\grD$ are measurable subsets of $[0,1)$, we put
$$R_1(u;\grC)=\int_\grC G(\alp)e(\alp u)\d\alp \quad \text{and}\quad R_2(v;\grD)=\int_\grD H(\bet)e(\bet v)\d\bet .$$
One then obtains
\begin{equation}\label{3.14}
N(B;\grC\times \grD)=\sum_{(u,v)\in \grX^2}\rho(u,v)R_1(u;\grC)R_2(v;\grD).
\end{equation}
Writing
\begin{equation}\label{3.15}
N_0(B;\grC,\grD)=\sum_{(u,v)\in \grX_0}\rho(u,v)R_1(u;\grC)R_2(v;\grD),
\end{equation}
the starting point for our analysis is the lower bound
\begin{equation}\label{3.16}
N(B)\ge N_0(B;[0,1),[0,1)).
\end{equation}
Our Hardy-Littlewood dissection is now executed by disassembling the set $[0,1)\times [0,1)$ into the four pieces
$$\grM\times \grM,\quad \grM\times \grm,\quad \grm\times \grM\quad \text{and}\quad \grm\times \grm.$$
We examine each of these subsets in turn. 

\begin{lemma}\label{lemma3.3} For systems of type A, one has
$$N_0(B;\grM,\grM)-N(B;\grM\times \grM)\ll B^{s-6-\tau}.$$
\end{lemma}

\begin{proof} By applying the triangle inequality in combination with Lemma \ref{lemma2.3}, we find that
\begin{equation}\label{3.17}
R_1(u;\grM)\ll B^{m-3+\eps}\quad \text{and}\quad R_2(v;\grM)\ll B^{n-3+\eps}.
\end{equation}
Consequently, on recalling (\ref{3.12}), (\ref{3.14}) and (\ref{3.15}), and then applying Lemma \ref{lemma3.2}, we deduce that
\begin{align*}
N(B;\grM\times \grM)-N_0(B;\grM,\grM)&=\sum_{(u,v)\in \grX_1\cup \grX_2}\rho(u,v)R_1(u;\grM)R_2(v;\grM)\notag \\
&\ll B^{s-l-6+\eps}(\Xi_1+\Xi_2)\ll B^{s-6-9\tau+\eps}.
\end{align*}
This completes the proof of the lemma.
\end{proof}

\begin{lemma}\label{lemma3.4} For systems of type A, one has
$$N_0(B;\grM,\grm)\ll B^{s-6-\tau}\quad \text{and}\quad N_0(B;\grm,\grM)\ll B^{s-6-\tau}.$$
\end{lemma}

\begin{proof} 
Applying the first upper bound of (\ref{3.17}) in concert with (\ref{3.10}), one sees that
\begin{align}
N_0(B;\grM,\grm)&\ll B^{m-3+\eps}\sum_{(u,v)\in \grX_0}\rho(u,v)|R_2(v;\grm)|\notag \\
&\le B^{m-3+\eps}\sum_{v\in \grX}\rho_2(v)|R_2(v;\grm)|.\label{3.18}
\end{align}
On the one hand, by applying Bessel's inequality together with the first estimate of Lemma \ref{lemma2.4}, one has
$$\sum_{v\in \grX}\rho_2(v)^2=\int_0^1|F(0,\bet)|^2\d\bet \ll B^{2l-11/4-9\tau}.$$
On the other hand, Bessel's inequality in combination with the second estimate of Lemma \ref{lemma2.4} yields the bound
$$\sum_{v\in \grX}|R_2(v;\grm)|^2\le \int_\grm|H(\bet)|^2\d\bet\ll B^{2n-13/4-8\tau}.$$
Consequently, by applying Cauchy's inequality to (\ref{3.18}), we obtain
$$N_0(B;\grM,\grm)\ll B^{m-3+\eps}\Bigl( \sum_{v\in \grX}\rho_2(v)^2\Bigr)^{1/2}\Bigl( \sum_{v\in \grX}|R_2(v;\grm)|^2\Bigr)^{1/2}\ll B^{s-6-8\tau}.$$
A symmetrical argument shows similarly that $N_0(B;\grm,\grM)\ll B^{s-6-8\tau}$, and thus the proof of the lemma is complete.
\end{proof}

\begin{lemma}\label{lemma3.5} For systems of type A, one has $N_0(B;\grm,\grm)\ll B^{s-6-\tau}$.
\end{lemma}

\begin{proof} Recall that for systems of type A, we take $T=B^{l-11/4}$ for the truncation parameter. First applying Cauchy's inequality, and then applying (\ref{3.10}) and (\ref{3.11}), therefore, in our first step we deduce that
\begin{align*}
N_0(B;\grm,\grm)&\le \Bigl( \sum_{(u,v)\in \grX_0}\rho(u,v)|R_1(u;\grm)|^2\Bigr)^{1/2}\Bigl( \sum_{(u,v)\in \grX_0}\rho(u,v)|R_2(v;\grm)|^2\Bigr)^{1/2}\\
&\le \Bigl( \sum_{\substack{u\in \grX\\ \rho_1(u)\le B^{l-11/4}}}\!\!\!\!\!\!\rho_1(u)|R_1(u;\grm)|^2\Bigr)^{1/2}\Bigl( \sum_{\substack{v\in \grX\\ \rho_2(v)\le B^{l-11/4}}}\!\!\!\!\!\!\rho_2(v)|R_2(v;\grm)|^2\Bigr)^{1/2}\\
&\le B^{l-11/4}\Bigl( \sum_{u\in \grX}|R_1(u;\grm)|^2\Bigr)^{1/2}\Bigl( \sum_{v\in \grX}|R_2(v;\grm)|^2\Bigr)^{1/2}.
\end{align*}
Next, applying Bessel's inequality together with Lemma 2.4, we conclude that
\begin{align*}
N_0(B;\grm,\grm)&\le B^{l-11/4}\Bigl( \int_\grm |G(\alp)|^2\d\alp \Bigr)^{1/2}\Bigl( \int_\grm |H(\bet)|^2\d\bet \Bigr)^{1/2}\\
&\ll B^{l-11/4}(B^{2m-13/4-8\tau})^{1/2}(B^{2n-13/4-8\tau})^{1/2}\ll B^{s-6-8\tau}.
\end{align*}
This completes the proof of the lemma.
\end{proof}

We now come to the crescendo of our argument for systems of type A. Combining the upper bounds provided by Lemmata \ref{lemma3.3}, \ref{lemma3.4} and \ref{lemma3.5}, we deduce from (\ref{3.16}) that
\begin{align*}
N(B)&\ge N_0(B;\grM,\grM)+N_0(B;\grM,\grm)+N_0(B;\grm,\grM)+N_0(B;\grm,\grm)\\
&=N(B;\grM\times \grM)+O(B^{s-6-\tau}).
\end{align*}
Hence, in view of (\ref{2.6}), we conclude from Lemmata \ref{lemma2.1} and \ref{lemma3.1} that
$$N(B)\ge N(B;\grN)+O(B^{s-6}\calL^{-1})\gg B^{s-6}.$$
This completes the proof of Theorem \ref{theorem1.1} for systems of type A.

\section{Systems of type B} The proof of Theorem \ref{theorem1.1} in situations wherein $n=3$ is complicated by the relative inferiority of the minor arc bounds available for $H(\bet)$ in mean square. Our argument for systems of type B, in which $l\ge 4$, $m\ge 4$ and $n\ge 3$, though modelled on that of the previous section, must therefore be modified in order to exploit better the exceptional nature of elements in the sets $\grX_1$ and $\grX_2$. Since Lemma \ref{lemma2.1} remains valid, and shows that $N(B;\grN)\gg B^{s-6}$, our first goal is to show that the conclusion of Lemma \ref{lemma3.1} remains valid in the present circumstances.\par

\begin{lemma}\label{lemma4.1} For systems of type B, one has $N(B;\grK)\ll B^{s-6}\calL^{-1}$.
\end{lemma}

\begin{proof} We begin by deriving an auxiliary mean value estimate. When $j\in \{l-1,l\}$, define
$$F_j(\alp,\bet)=h(a_1\alp+b_1\bet)h(a_2\alp+b_2\bet)\prod_{\substack{3\le i\le l\\ i\ne j}}g(a_i\alp+b_i\bet).$$
Then as a consequence of Schwarz's inequality, one has
$$\int_0^1|F_j(\alp,\bet)H(\bet)|\d\bet \le \Bigl( \int_0^1|F_j(\alp,\bet)|^2\d\bet\Bigr)^{1/2}\Bigl( \int_0^1|H(\bet)|^2\d\bet\Bigr)^{1/2}.$$
By orthogonality, the first integral on the right hand side here is bounded above by the number of solutions of a diophantine equation, and so by applying Lemma \ref{2.4} we obtain
$$\int_0^1|F_j(\alp,\bet)|^2\d\bet \le \int_0^1|F_j(0,\bet)|^2\d\bet \ll B^{2(l-1)-11/4-9\tau}.$$
The second integral on the right hand side may also be estimated via Lemma \ref{lemma2.4}, so that
$$\int_0^1|H(\bet)|^2\d\bet \ll B^{2n-11/4-9\tau}.$$
We therefore deduce that
$$\int_0^1|F_j(\alp,\bet)H(\bet)|\d\bet \ll B^{l+n-15/4-9\tau}.$$

\par Our next step is to prune the set $\grK$, the better to exploit available major arc estimates. Define
\begin{align*}
\grK_0&=\{(\alp,\bet)\in\grK:\text{$a_{l-1}\alp+b_{l-1}\bet\in \grM$ and $a_l\alp+b_l\bet\in \grM$}\},\\ 
\grk_i(\alp)&=\{ \bet\in [0,1):a_i\alp+b_i\bet\in \grm\}\quad (i=l-1,l).
\end{align*}
Then from the modified version of Weyl's inequality (\ref{2.11}), when $i\in \{l-1,l\}$ one sees that
$$\sup_{\bet\in \grk_i(\alp)}|g(a_i\alp+b_i\bet)|\le \sup_{\tet\in \grm}|g(\tet)|\ll B^{3/4+\eps}.$$
Uniformly in $\alp$, therefore, one has the estimate
\begin{align*}
\int_{\grk_i(\alp)}|F(\alp,\bet)H(\bet)|\d\bet&\le \Bigl(\sup_{\bet\in\grk_i(\alp)}|g(a_i\alp+b_i\bet)|\Bigr)\int_0^1|F_i(\alp,\bet)H(\bet)|\d\bet\\
&\ll B^{3/4+\eps}(B^{l+n-15/4-9\tau})\ll B^{s-m-3-8\tau}.
\end{align*}
Consequently, on recalling Lemma \ref{lemma2.3}, one discerns the upper bound
\begin{align*}
N(B;\grK\setminus \grK_0)&\le \sum_{i=l-1}^l\int_\grM |G(\alp)|\int_{\grk_i(\alp)}|F(\alp,\bet)H(\bet)|\d\bet\d\alp \\
&\ll B^{s-m-3-8\tau}\int_\grM |G(\alp)|\d\alp \ll B^{s-6-\tau}.
\end{align*}
In this way, we deliver the interim conclusion
\begin{equation}\label{4.1}
N(B;\grK)=N(B;\grK_0)+O(B^{s-6-\tau}).
\end{equation}

\par We now imitate the argument of the proof of Lemma \ref{lemma3.1}, employing notation from the latter proof with some minor modifications. We define $\Phi_G(\alp)$ as in (\ref{3.2}), and modify the definition of $\Phi_H(\bet)$ by putting
$$\Phi_H(\bet)=|g(d_n\bet)|^{5/2}|h(d_1\bet)h(d_2\bet)|.$$
One finds with little effort that the estimates $I_G\ll B^{3/2}$ and $I_H\ll B^{3/2}$ remain valid in present circumstances. Defining $J_{i,k}^G$ and $J_{i,k}^H$ as in (\ref{3.5}), though noting our revised definition of the set $\grK_0$, one finds just as in the argument leading to (\ref{3.6}) that when $i\in\{1,2\}$ and $k\in \{l-1,l\}$, one has $J_{i,k}^G\ll B^3$ and $J_{i,k}^H\ll B^3$. In the current situation, we modify the definition of $\Psi(\alp,\bet)$ by putting
$$\Psi(\alp,\bet)=\prod_{i=1}^2|h(c_i\alp)h(d_i\bet)^3h(a_i\alp+b_i\bet)|.$$
The reader will have no difficulty in confirming that the upper bound (\ref{3.7}) remains valid. Consequently, an application of H\"older's inequality reveals that
$$N(B;\grK_0)\le g(0)^{s-11}\Bigl(\sup_{(\alp,\bet)\in\grn}\Psi(\alp,\bet)\Bigr)^{1/5} (J_{1,l}^GJ_{2,l-1}^G)^{3/10}(J_{1,l}^HJ_{2,l-1}^H)^{1/10}J_0^{1/5},$$
in which $J_0$ is defined by (\ref{3.4}). The upper bound $J_0\ll B^3$, combined with our earlier estimates, therefore leads to the asymptotic relation
$$N(B;\grK_0)\ll B^{s-11}(B^{10}Q^{-1/10})^{1/5}(B^6)^{3/10}(B^6)^{1/10}(B^3)^{1/5}=B^{s-6}Q^{-1/50}.$$
The conclusion of the lemma is now confirmed by recalling (\ref{4.1}).
\end{proof}

At this stage of our discussion we introduce unconventional elements paralleling those introduced in the preambles to Lemmata \ref{lemma3.2} and \ref{lemma3.3}, employing the same notation throughout. We have only to record that for systems of type B, the truncation parameter is fixed to be $T=B^{l-3+\tau}$. Before launching the Hardy-Littlewood dissection proper, we pause to establish an auxiliary estimate for the quantity
$$\Zet_i=\sum_{v\in \grX}\Biggl( \sum_{\substack{u\in \grX\\ (u,v)\in \grX_i}}\rho(u,v)\Biggr)^2\quad (i=1,2).$$

\begin{lemma}\label{lemma4.2} For systems of type B, one has $\Zet_i\ll B^{2l-13/4-8\tau}$ $(i=1,2)$.
\end{lemma}

\begin{proof} We first seek to establish the lemma in the case $i=1$. Suppose that $u\in \grX$ is an integer for which $\rho_1(u)>B^{l-3+\tau}$. Then one has
$$\int_0^1F(\alp,0)e(-u\alp)\d\alp =\rho_1(u)>B^{l-3+\tau}.$$
For systems of type B one has $l\ge 4$. As in the argument of the proof of Lemma \ref{lemma2.3}, one therefore finds from Lemma \ref{lemma2.2} that
$$\Bigl|\int_\grM F(\alp,0)e(-u\alp)\d\alp\Bigr|\le \int_\grM |F(\alp,0)|\d\alp \ll B^{l-3+\eps},$$
whence
$$\Bigl|\int_\grm F(\alp,0)e(-u\alp)\d\alp \Bigr|\ge \tfrac{1}{2}\rho_1(u).$$
In this way, one obtains the upper bound
\begin{align*}
\sum_{(u,v)\in \grX_1}\rho(u,v)&\le \sum_{\substack{u\in \grX\\ \rho_1(u)>B^{l-3+\tau}}}\sum_{v\in \grX}\rho(u,v)\le B^{3-l-\tau}\sum_{\substack{u\in\grX\\ \rho_1(u)>B^{l-3+\tau}}}\rho_1(u)^2\\
&\ll B^{3-l-\tau}\sum_{u\in\grX}\Bigl| \int_\grm F(\alp,0)e(-u\alp)\d\alp \Bigr|^2.
\end{align*}
From here, an application of Bessel's inequality in combination with the second bound of Lemma \ref{lemma2.4} yields
$$\sum_{(u,v)\in \grX_1}\rho(u,v)\ll B^{3-l-\tau}\int_\grm |F(\alp,0)|^2\d\alp \ll B^{3-l-\tau}(B^{2l-13/4-8\tau}).$$
However, when $(u,v)\in \grX_1$ one has
$$\sum_{\substack{u\in \grX\\ (u,v)\in \grX_1}}\rho(u,v)\le \rho_2(v)\le B^{l-3+\tau},$$
and so one arrives at the upper bound
$$\sum_{v\in \grX}\Bigl( \sum_{\substack{u\in \grX\\ (u,v)\in \grX_1}}\rho(u,v)\Bigr)^2\le B^{l-3+\tau}\sum_{(u,v)\in \grX_1}\rho(u,v)\ll B^{l-3+\tau}(B^{l-1/4-9\tau}).$$
The conclusion of the lemma has therefore been established when $i=1$.\par

When $i=2$, we follow a similar though simpler path. Thus, one obtains
\begin{align*}
\sum_{v\in \grX}\Bigl( \sum_{\substack{u\in \grX\\ (u,v)\in \grX_2}}\rho(u,v)\Bigr)^2&\le \sum_{\substack{v\in \grX\\ \rho_2(v)>B^{l-3+\tau}}}\rho_2(v)^2\ll \sum_{v\in \grX}\Bigl| \int_\grm F(0,\bet)e(-\bet v)\d\bet \Bigr|^2\\
&\le \int_\grm |F(0,\bet)|^2\d\bet \ll B^{2l-13/4-8\tau}.
\end{align*}
This completes the proof of the lemma in the case $i=2$.
\end{proof}

In the present circumstances, our Hardy-Littlewood dissection proceeds by disassembling the set $[0,1)\times [0,1)$ into the three pieces
$$\grM\times \grM, \quad \grm\times [0,1)\quad \text{and}\quad \grM\times \grm.$$
We analyse these subsets in turn by means of three lemmata.

\begin{lemma}\label{lemma4.3} For systems of type B, one has
$$N_0(B;\grM,\grM)-N(B;\grM\times \grM)\ll B^{s-6-\tau}.$$
\end{lemma}

\begin{proof} We begin by deriving an auxiliary estimate for the quantity
$$\Ups_i=\sum_{(u,v)\in \grX_i}\rho(u,v)|R_2(v;\grM)|\quad (i=1,2).$$
Observe that by applying Bessel's inequality in combination with the first estimate of Lemma \ref{lemma2.4}, one discerns that
$$\sum_{v\in \grX}|R_2(v;\grM)|^2\le \int_0^1|H(\bet)|^2\d\bet \ll B^{2n-11/4-9\tau}.$$
When $i\in \{1,2\}$, therefore, we deduce from Cauchy's inequality together with Lemma \ref{lemma4.2} that
$$\Ups_i\le \Zet_i^{1/2}\Bigl( \sum_{v\in \grX}|R_2(v;\grM)|^2\Bigr)^{1/2}\ll B^{l+n-3-8\tau}.$$
For systems of type B one has $m\ge 4$, and so it follows from Lemma \ref{lemma2.3} that $R_1(u;\grM)\ll B^{m-3+\eps}$. Hence, we obtain
\begin{align*}
N(B;\grM\times \grM)-N_0(B;\grM,\grM)&=\sum_{(u,v)\in \grX_1\cup \grX_2}\rho(u,v)R_1(u;\grM)R_2(v;\grM)\\
&\ll B^{m-3+\eps}(\Ups_1+\Ups_2)\ll B^{s-6-7\tau }.
\end{align*}
This completes the proof of the lemma.
\end{proof}

\begin{lemma}\label{lemma4.4} For systems of type B, one has $N_0(B;\grm,[0,1))\ll B^{s-6-\tau}$.
\end{lemma}

\begin{proof} An application of Cauchy's inequality reveals that
$$N_0(B;\grm,[0,1))\le V_1^{1/2}V_2^{1/2},$$
where
$$V_1=\sum_{(u,v)\in \grX_0}\rho(u,v)|R_1(u;\grm)|^2$$
and
$$V_2=\sum_{(u,v)\in\grX_0}\rho(u,v)|R_2(v;[0,1))|^2.$$
On recalling the definitions of the sets $\grX_i$ from (\ref{3.11}), noting that at present $T=B^{l-3+\tau}$, it follows from Bessel's inequality and Lemma \ref{lemma2.4} that
\begin{align*}V_1&\le \sum_{\substack{u\in \grX\\ \rho_1(u)\le B^{l-3+\tau}}}\rho_1(u)|R_1(u;\grm)|^2\le B^{l-3+\tau}\sum_{u\in \grX}|R_1(u;\grm)|^2\\
&\le B^{l-3+\tau}\int_\grm |G(\alp)|^2\d\alp \ll B^{l-3+\tau}(B^{2m-13/4-8\tau}).
\end{align*}
Similarly, one finds that
\begin{align*}
V_2&\le \sum_{\substack{v\in \grX\\ \rho_2(v)\le B^{l-3+\tau}}}\rho_2(v)|R_2(v;[0,1))|^2\le B^{l-3+\tau}\sum_{v\in \grX}|R_2(v;[0,1))|^2\\
&\le B^{l-3+\tau}\int_0^1|H(\bet)|^2\d\bet \ll B^{l-3+\tau}(B^{2n-11/4-9\tau}).
\end{align*}
Thus we deduce that
$$N_0(B;\grm,[0,1))\ll B^{l-3+\tau}(B^{n+m-3-8\tau})\le B^{s-6-7\tau},$$
and the proof of the lemma is complete.
\end{proof}

\begin{lemma}\label{lemma4.5} For systems of type B, one has $N_0(B;\grM,\grm)\ll B^{s-6-\tau}$.
\end{lemma}

\begin{proof} Following the argument of the proof of Lemma \ref{lemma4.3}, one finds that
\begin{align*}
N(B;\grM\times \grm)-N_0(B;\grM,\grm)&=\sum_{(u,v)\in \grX_1\cup \grX_2}\rho(u,v)R_1(u;\grM)R_2(v;\grm)\notag \\
&\ll B^{m-3+\eps}(\Zet_1+\Zet_2)^{1/2}\Bigl( \int_0^1|H(\bet)|^2\d\bet \Bigr)^{1/2}.
\end{align*}
Thus we deduce that
\begin{equation}\label{4.2}
N_0(B;\grM,\grm)-N(B;\grM\times \grm)\ll B^{s-6-7\tau}.
\end{equation}

\par We next observe that
\begin{equation}\label{4.3}
N(B;\grM\times \grm)=\int_\grM\int_\grm F(\alp,\bet)G(\alp)H(\bet)\d\bet\d\alp .
\end{equation}
As a consequence of Schwarz's inequality, one has
$$\int_0^1|F(\alp,\bet)h(d_1\bet)h(d_2\bet)|\d\bet\le g(0)^{l-4}U_1^{1/2}U_2^{1/2},$$
where for $i\in \{1,2\}$ we write
$$U_i=\int_0^1|g(a_{2+i}\alp+b_{2+i}\bet)h(a_i\alp+b_i\bet)h(d_i\bet)|^2\d\bet .$$
On considering the underlying Diophantine equations and then appealing to Lemma \ref{lemma2.4}, one discerns that
$$U_i\le \int_0^1|g(b_{2+i}\bet)h(b_i\bet)h(d_i\bet)|^2\d\bet \ll B^{13/4-9\tau}.$$
We therefore deduce from the modified version of Weyl's inequality (\ref{2.11}) that
\begin{align*}
\int_\grm |F(\alp,\bet)H(\bet)|\d\bet &\le g(0)^{n-3}\Bigl(\sup_{\bet\in \grm}|g(d_n\bet)|\Bigr) \int_0^1|F(\alp,\bet)h(d_1\bet)h(d_2\bet)|\d\bet \\
&\ll B^{n-3}(B^{3/4+\eps})(B^{l-3/4-9\tau})\ll B^{s-m-3-8\tau}.
\end{align*}
Substituting this upper bound into (\ref{4.3}) and applying Lemma \ref{lemma2.3}, we obtain
$$N(B;\grM\times \grm)\ll B^{s-m-3-8\tau}\int_\grM |G(\alp)|\d\alp \ll B^{s-6-7\tau}.$$
The conclusion of the lemma follows by reference to (\ref{4.2}).
\end{proof}

We are now equipped to finish off the discussion of systems of type B. Combining the estimates supplied by Lemmata \ref{lemma4.3}, \ref{lemma4.4} and \ref{lemma4.5}, we see that
\begin{align*}
N(B)&\ge N_0(B;\grM,\grM)+N_0(B;\grM,\grm)+N_0(B;\grm,[0,1))\\
&=N(B;\grM\times\grM)+O(B^{s-6-\tau}).
\end{align*}
Hence, in view of (\ref{2.6}), we conclude from Lemmata \ref{lemma2.1} and \ref{lemma4.1} that
$$N(B)\ge N(B;\grN)+O(B^{s-6}\calL^{-1})\gg B^{s-6}.$$
This completes the proof of Theorem \ref{theorem1.1} for systems of type B.

\section{Systems of type C} Our analysis of systems of type C, wherein $s=12$ and $(l,m,n)=(2,5,5)$, may be abbreviated by adjusting the argument of \S3 through modification of the generating functions $F(\alp,\bet)$, $G(\alp)$ and $H(\bet)$. We begin with a discussion of the pruning operation implicit in the estimation of $N(B;\grK)$.\par

\begin{lemma}\label{lemma5.1}
For systems of type C, one has $N(B;\grK)\ll B^6\calL^{-1}$.
\end{lemma}

\begin{proof} Define the mean values
\begin{align*}
U_{ij}&=\iint_\grK |g(c_i\alp)|^{5/2}|g(d_j\bet)|^{9/2}|h(a_1\alp+b_1\bet)|^2\d\alp\d\bet ,\\
V_{ij}&=\iint_\grK |g(c_i\alp)|^{9/2}|g(d_j\bet)|^{5/2}|h(a_2\alp+b_2\bet)|^2\d\alp\d\bet ,\\
W_k&=\int_0^1\int_0^1|h(c_k\alp)h(d_k\bet)|^8\d\alp\d\bet ,
\end{align*}
and put
$$\Psi(\alp,\bet)=|h(c_1\alp)h(c_2\alp)h(d_1\bet)h(d_2\bet)|^3|h(a_1\alp+b_1\bet)h(a_2\alp+b_2\bet)|.$$
Then an application of H\"older's inequality reveals that
\begin{equation}\label{5.1}
N(B;\grK)\le \Bigl( \sup_{(\alp,\bet)\in \grn}\Psi(\alp,\bet)\Bigr)^{1/7}(W_1W_2)^{1/14}\prod_{i=3}^5\prod_{j=3}^5(U_{ij}V_{ij})^{1/21}.
\end{equation}

\par The argument of the proof of \cite[Lemma 10]{BW2007a} shows that
$$\sup_{(\alp,\bet)\in \grn}\Psi(\alp,\bet)\ll B^{14}Q^{-1/10}.$$
As a consequence of Lemma \ref{lemma2.2}, meanwhile, one has
$$U_{ij}\le \Bigl( \int_\grM|g(d_j\bet)|^{9/2}\d\bet\Bigr) \Bigl( \sup_{\lam\in\dbR}\int_\grM|g(c_i\alp)|^{5/2}|h(a_1\alp+\lam)|^2\d\alp\Bigr)\ll B^3,$$
and a symmetric argument yields the estimate $V_{ij}\ll B^3$. Finally, one finds from \cite[Theorem 2]{Vau1986} that
$$W_k\le \Bigl( \int_0^1|h(\tet)|^8\d\tet\Bigr)^2\ll (B^5)^2=B^{10}.$$
Combining these estimates within (\ref{5.1}), we conclude that
$$N(B;\grK)\ll (B^{14}Q^{-1/10})^{1/7}(B^{20})^{1/14}(B^6)^{9/21}\ll B^6Q^{-1/70}.$$
This completes the proof of the lemma.
\end{proof}

Our next step is to relabel the coefficients of the system (\ref{1.1}) so that $\mtil=m-1$, $\ntil=n-1$, $\ltil=l+2$, which is to say that $(\ltil,\mtil,\ntil)=(4,4,4)$, and to put
$$\ctil_j=c_j\quad \text{and}\quad \dtil_j=d_j\quad (1\le j\le 4),$$
and
$$(\atil_i,\btil_i)=(a_i,b_i)\quad (i=1,2),\quad (\atil_3,\btil_3)=(0,d_5)\quad (\atil_4,\btil_4)=(c_5,0).$$
We then define the generating functions $\Ftil(\alp,\bet)$, $\Gtil(\alp)$ and $\Htil(\bet)$ as in the respective definitions of $F(\alp,\bet)$, $G(\alp)$ and $H(\bet)$ in (\ref{2.2}) and (\ref{2.3}), save that in the present context the integers $l$, $m$, $n$, and the coefficients $a_i$, $b_i$, $c_j$ and $d_k$, are to be decorated by tildes. Further notation from \S\S2 and 3 is understood to have the meaning naturally inferred in like manner when decorated by a tilde. An examination of the argument of \S3, leading from the discussion preceding Lemma \ref{lemma3.2} to the conclusion of the section, now reveals that no adjustment is necessary in order to accommodate the change of circumstances implicit in our present analysis. Here it is worth noting that, despite the fact that we now have $\ltil=4$ and $\atil_3=0$, the presence of three non-zero coefficients in the equation (\ref{3.8}) ensures that the analogue of the upper bound (\ref{3.13}) remains valid. Thus one obtains $\Xitil_1\ll B^{4-9\tau}$, and by means of a symmetric argument also $\Xitil_2\ll B^{4-9\tau}$. The analogue of Lemma \ref{lemma3.3} delivers the bound
$$\Ntil_0(B;\grM,\grM)-\Ntil(B;\grM\times \grM)\ll B^{6-\tau},$$
and analogues of Lemmata \ref{lemma3.4} and \ref{lemma3.5} yield the estimates
$$\Ntil_0(B;\grM,\grm)\ll B^{6-\tau},\quad \Ntil_0(B;\grm,\grM)\ll B^{6-\tau},\quad \Ntil_0(B;\grm,\grm)\ll B^{6-\tau}.$$
We therefore conclude that
\begin{align*}
\Ntil(B)&\ge \Ntil_0(B;\grM,\grM)+\Ntil_0(B;\grM,\grm)+\Ntil_0(B;\grm,\grM)+\Ntil_0(B;\grm,\grm)\\
&=\Ntil(B;\grM\times \grM)+O(B^{6-\tau})=N(B;\grM\times \grM)+O(B^{6-\tau}).
\end{align*}
Then, in view of (\ref{2.6}), we conclude from Lemmata \ref{lemma2.1} and \ref{lemma5.1} that
$$N(B)=\Ntil(B)\gg N(B;\grN)+O(B^6\calL^{-1})\gg B^6.$$
This completes the proof of Theorem \ref{theorem1.1} for systems of type C.

\section{Systems of type D} At present, we have been unable to devise an unconditional treatment of systems of the shape (\ref{1.1}) in which $(l,m,n)=(5,5,2)$. A conditional treatment is available by appealing to HRH. In order to describe the nature of this particular Riemann Hypothesis, we must indulge in some discussion. Although a lengthy affair in full, for the sake of concision we conduct a rather sketchy account here of the treatment of systems of type D. Following Hooley \cite[\S\S5 and 6]{Hoo1986}, we consider the cubic form $\grg(\bfx)=x_1^3+\ldots +x_6^3$ and the associated discriminant
$$\Del(\bfm)=3\prod (m_1^{3/2}\pm m_2^{3/2}\pm \ldots \pm m_6^{3/2}),$$
in which the product is taken over all possible choices of the signs. Let $\rho(\bfm;p^r)$ denote the number of points of the projective variety defined by $\grg(\bfx)=\bfm\cdot \bfx=0$, having coordinates in the finite field $\dbF_{p^r}$, and put
\begin{equation}\label{6.1}
E(\bfm;p^r)=\rho(\bfm;p^r)-(p^{4r}-1)/(p^r-1).
\end{equation}
The Euler factors $L_p(\bfm;s)$ are then defined for $p\nmid \Del(\bfm)$ by putting
$$L_p(\bfm;s)=\exp \Bigl( -\sum_{r=1}^\infty E(\bfm;p^r)p^{-rs}/r\Bigr).$$
When $p|\Del(\bfm)$, one must modify the definition of $L_p(\bfm;s)$, as described by Serre \cite{Ser1986}, so that for suitable coefficients $\lam_{j,p}=\lam_{j,p}(\bfm)$ with $1\le |\lam_{j,p}|\le p^{3/2}$, one has
$$L_p(\bfm;s)=\prod_j (1-\lam_{j,p}p^{-s})^{-1}.$$
The number of factors here is at most $10$, the precise definition of which need not detain us. Associated to the modified Hasse-Weil $L$-function
$$L(\bfm;s)=\prod_pL_p(\bfm;s)$$
is the conductor $B(\bfm)$, given by
$$B(\bfm)=\prod_{p|\Del(\bfm)}p^{a_p},$$
in which the exponents $a_p$ are certain non-negative integers with $0\le a_p\le 200$. Finally, we put
$$\xi(\bfm;s)=(2\pi)^{-5s}\Gam(s)^5B(\bfm)^{s/2}L(\bfm;s).$$

\begin{conjecture}[HRH]\label{conjecture6.1} Suppose that $\Del(\bfm)\ne 0$. Then:
\begin{itemize}
\item[(i)] the function $\xi(\bfm;s)$ has a meromorphic continuation to $\dbC$ of finite order, its only possible poles being at $s=\frac{3}{2}$ and $s=\frac{5}{2}$;
\item[(ii)] with $w(\bfm)=\pm 1$, one has the functional equation
$$\xi(\bfm;s)=w(\bfm)\xi(\bfm;4-s);$$
\item[(iii)] when $\text{Re}(s)\ne 2$, one has $\xi(\bfm;s)\ne 0$.
\end{itemize}
\end{conjecture}

It is the assertion (iii) of this conjecture that constitutes the Riemann Hypothesis within HRH. The relevance of Conjecture \ref{conjecture6.1} for our work here is made visible by the following lemma.

\begin{lemma}\label{lemma6.2} Provided that {\rm HRH} be valid, one has
$$\int_0^1|f(\tet)|^6\d\tet \ll B^{3+\eps}.$$
\end{lemma}

\begin{proof} When $n$ is a non-negative integer, write $r(n)$ for the number of representations of $n$ as the sum of three non-negative integral cubes. Then, subject to the validity of HRH, Hooley \cite[\S5]{Hoo1996} has shewn that
\begin{equation}\label{6.2}
\sum_{1\le n\le x}r(n)^2\ll x^{1+\eps},
\end{equation}
a conclusion that yields the bound claimed in the lemma as an immediate corollary. We note that Heath-Brown \cite[Theorem 1.1]{HB1998} has also shown that the upper bound (\ref{6.2}) holds conditional on the truth of HRH\footnote{In order to avoid possible confusion, we note that in the display preceding \cite[equation (4.4)]{HB1998}, there is a typographic error which is corrected in (\ref{6.1}) above.}.
\end{proof}

Henceforth in this section, we assume the truth of HRH. We now put $Y=B^{10\tau}$, and introduce the generating functions
$$k_p(\tet)=\sum_{B/p<w\le 2B/p}e(\tet w^3)\quad \text{and}\quad K(\tet ;Y)=\sum_{Y<p\le 2Y}k_p(p^3\tet),$$
in which the letter $p$ is reserved to indicate a prime number in the congruence class $2$ modulo $3$. Finally, we change the definition of the generating function $h(\tet)$ applied hitherto by setting
\begin{equation}\label{6.3}
h(\tet)=\sum_{j=1}^JK(\tet;2^{-j}Y),
\end{equation}
where $J=[\frac{1}{2}\tau \log B]$.

\begin{lemma}\label{lemma6.3} When $a$ and $b$ are non-zero integers, one has
$$\int_\grm |g(a\tet)^4h(b\tet)^6|\d\tet \ll B^{6-3\tau}.$$
\end{lemma}

\begin{proof} We apply the argument of the proof of Theorem 3.1 of the authors' recent work \cite{BW2010} concerning sums of cubes and minicubes, substituting the conditional bound supplied by Lemma \ref{lemma6.2} in place of the bound tantamount to (\ref{2.1}) employed in \cite{BW2010}. In the first instance, the relevance of this new bound is seen on considering the underlying Diophantine equations. One finds that
$$\int_0^1|K(\tet;Y)|^6\d\tet \ll B^{3+\eps},$$
and likewise
$$\int_0^1\Bigl( \max_{Y<p\le 2Y}|k_p(\tet)|\Bigr)^6\d\tet \ll (B/Y)^{3+\eps}.$$

\par Next, when $X$ is a real parameter with $1\le X\le B^{3/2}$, define
$$\grM(q,a;X)=\{ \tet \in [0,1):|q\tet-a|\le XB^{-3}\},$$
and then take $\grM(X)$ to be the union of the arcs $\grM(q,a;X)$ with $0\le a\le q\le X$ and $(a,q)=1$. We put $\grm(X)=[0,1)\setminus\grM(X)$. In view of our choice for $Y$, it follows by an application of H\"older's inequality paralleling that employed in the proof of \cite[Corollary 3.2]{BW2010}, that
$$\int_{\grm(BY^3)}|g(a\tet)^2h(b\tet)^6|\d\tet \ll B^{9/2+\eps}(Y/B^\tau)^{-1/2}\ll B^{9/2-4\tau}.$$
As a consequence of Weyl's inequality (see \cite[Lemma 2.4]{Vau1997}), moreover, one has
$$\sup_{\tet \in \grm(BY^3)}|g(a\tet)|\ll B^{3/4+\eps}.$$
Hence we deduce that
\begin{equation}\label{6.4}
\int_{\grm(BY^3)}|g(a\tet)^4h(b\tet)^6|\d\tet \ll (B^{3/4+\eps})^2B^{9/2-4\tau}\ll B^{6-3\tau}.
\end{equation}

\par We next prune from $\grm$ to $\grm(BY^3)$. By a modified version of Weyl's inequality akin to that embodied in (\ref{2.11}), one finds that whenever $Y<p\le 2Y$, then
$$\sup_{\tet\in \grm}|k_p(bp^3\tet)|\ll \sup_{\tet \in \grm (B^{3/4}Y^{-4})}|k_p(\tet)|\ll B^{3/4+\eps}Y^2,$$
whence
$$\sup_{\tet\in\grm}|h(b\tet)|\ll B^{3/4+\eps}Y^3L\ll B^{4/5}.$$
In addition, the methods of \cite[\S\S4.3 and 4.4]{Vau1997} permit one to establish the estimate
$$\int_{\grM(BY^3)}|g(a\tet)|^4\d\tet \ll B^{1+\eps}Y^{12}.$$
We are consequently led to the upper bound
\begin{align*}
\int_{\grm\setminus \grm(BY^3)}|g(a\tet)^4h(b\tet)^6|\d\tet &\ll \Bigl( \sup_{\alp\in\grm}|h(b\tet)|\Bigr)^6\int_{\grM(BY^3)}|g(a\tet)|^4\d\tet \\
&\ll (B^{4/5})^6B^{1+\eps}Y^{12}\ll B^{6-3\tau}.
\end{align*}
The conclusion of the lemma follows by reference to (\ref{6.4}).
\end{proof}

We now aim to follow the argument of \S4, making adjustments as necessary. We first revise the definitions of the generating functions in (\ref{2.2}) and (\ref{2.3}) by putting
$$F(\alp,\bet)=g(a_4\alp+b_4\bet)g(a_5\alp+b_5\bet)\prod_{i=1}^3h(a_i\alp+b_i\bet),$$
$$G(\alp)=g(c_4\alp)g(c_5\alp)\prod_{j=1}^3h(c_j\alp),\quad H(\bet)=g(d_1\bet)g(d_2\bet).$$
Defining $N(B;\grB)$ as in (\ref{2.4}), we again obtain the lower bound (\ref{2.5}) for $N(B)$. All other definitions remain unchanged in the discussion to follow, unless explicitly noted. Notice that the new definition (\ref{6.3}) of the generating function $h(\tet)$ ensures that its behaviour on major arcs is very nearly as congenial as that of $g(\tet)$, since it differs from a classical Weyl sum only by the presence of a small prime factor. Indeed, one may sum trivially over this factor whenever necessary, treating the remaining part as a classical Weyl sum. In this way, one may verify that the conclusions of Lemmata \ref{lemma2.1}, \ref{lemma2.2} and \ref{lemma2.3} remain valid in the current situation, despite the novel identity of the exponential sum $h(\tet)$.\par

Our next step is to derive an analogue of the pruning lemma of \S4.

\begin{lemma}\label{lemma6.4} For systems of type D, one has $N(B;\grK)\ll B^6\calL^{-1}$.
\end{lemma}

\begin{proof} Define the mean values
\begin{align*}
U_{ij}&=\iint_\grK |g(c_i\alp)g(d_i\bet)|^{19/9}|h(a_j\alp+b_j\bet)h(c_j\alp)|^2\d\alp\d\bet ,\\
W_j&=\int_0^1\int_0^1|h(c_j\alp)h(a_j\alp+b_j\bet)|^8\d\alp\d\bet ,
\end{align*}
and put
$$\Psi(\alp,\bet)=\prod_{j=1}^3|h(c_j\alp)h(a_j\alp+b_j\bet)|.$$
Then an application of H\"older's inequality reveals that
\begin{equation}\label{6.5}
N(B;\grK)\le g(0)^2\Bigl( \sup_{(\alp,\bet)\in \grn}\Psi(\alp,\bet)\Bigr)^{13/57}\prod_{j=1}^3(U_{4j}^9U_{5j}^9W_j)^{1/57}.
\end{equation}

\par The argument of the proof of \cite[Lemma 10]{BW2007a} shows that
$$\sup_{(\alp,\bet)\in \grn}\Psi(\alp,\bet)\ll B^6Q^{-1/10}.$$
As a consequence of Lemma \ref{lemma2.2}, meanwhile, one has
\begin{align*}
U_{ij}&\le \Bigl( \int_\grM |g(c_i\alp)|^{19/9}|h(c_j\alp)|^2\d\alp\Bigr) \Bigl( \sup_{\lam\in \dbR}\int_\grM |g(d_i\bet)|^{19/9}|h(b_j\bet+\lam)|^2\d\bet\Bigr)\\
&\ll (B^{10/9})^2=B^{20/9}.
\end{align*}
Also, by considering the underlying Diophantine equations, making a change of variables, and applying \cite[Theorem 2]{Vau1986}, one sees that
$$W_j\le \Bigl( \int_0^1|h(\tet)|^8\d\tet\Bigr)^2\ll (B^5)^2=B^{10}.$$
Combining these estimates within (\ref{6.5}), we conclude that
$$N(B;\grK)\ll B^2(B^6Q^{-1/10})^{13/57}((B^{20})^2B^{10})^{3/57}\ll B^6Q^{-13/570}.$$
The conclusion of the lemma now follows.
\end{proof}

We now proceed as in \S4, adopting the notation introduced in the discussion prior to Lemmata \ref{lemma3.2} and \ref{lemma3.3}. In present circumstances we have $(l,m,n)=(5,5,2)$, though comparison with \S4 will be assisted in what follows by explicit mention of $l$, $m$ and $n$. Thus, for systems of type D, the truncation parameter is fixed to be $T=B^{l-3+\tau}$, just as in \S4. Our next task is to derive a bound for the quantity $\Zet_i$ introduced in the preamble to Lemma \ref{lemma4.2}.

\begin{lemma}\label{lemma6.5} For systems of type D, one has $\Zet_i\ll B^{2l-4-3\tau}$ $(i=1,2)$.
\end{lemma}

\begin{proof} The reader will experience no difficulty in adapting the argument of the proof of Lemma \ref{lemma4.2} to establish the claimed bounds, substituting when needed the estimates
$$\int_\grm|F(\alp,0)|^2\d\alp\ll B^{2l-4-3\tau}\quad \text{and}\quad \int_\grm|F(0,\bet)|^2\d\bet\ll B^{2l-4-3\tau},$$
made available from Lemma \ref{lemma6.3} via H\"older's inequality.
\end{proof}

\begin{lemma}\label{lemma6.6} For systems of type D, one has
$$N_0(B;\grM,\grM)-N(B;\grM\times \grM)\ll B^{6-\tau}.$$
\end{lemma}

\begin{proof} We adapt the argument of the proof of Lemma \ref{lemma4.3} to the present context. First, as a consequence of Hua's lemma (see \cite[Lemma 2.5]{Vau1997}) and Schwarz's inequality, one has
\begin{align}
\int_0^1|H(\bet)|^2\d\bet &\le \Bigl( \int_0^1|g(d_1\bet)|^4\d\bet\Bigr)^{1/2}\Bigl( \int_0^1|g(d_2\bet)|^4\d\bet\Bigr)^{1/2}\notag  \\
&\ll B^{2+\eps}=B^{2n-2+\eps}.\label{6.6}
\end{align}
Then, as in the proof of Lemma \ref{lemma4.3}, we deduce from Lemma \ref{lemma6.5} that
$$\Ups_i\le \Zet_i^{1/2}\Bigl( \int_0^1|H(\bet)|^2\d\bet \Bigr)^{1/2}\ll B^{l+n-3-3\tau/2+\eps},$$
and hence that
$$N_0(B;\grM,\grM)-N(B;\grM\times \grM)\ll B^{m-3+\eps}(\Ups_1+\Ups_2)\ll B^{s-6-\tau}.$$
The desired conclusion follows on recalling that $s=12$.
\end{proof}

\begin{lemma}\label{lemma6.7} For systems of type D, one has $N_0(B;\grm,[0,1))\ll B^{6-\tau/3}$.
\end{lemma}

\begin{proof} By adapting the argument of the proof of Lemma \ref{lemma4.4}, one finds that
$$N_0(B;\grm,[0,1))\ll B^{l-3+\tau}\Bigl(\int_\grm |G(\alp)|^2\d\alp \Bigr)^{1/2}\Bigl( \int_0^1|H(\bet)|^2\d\bet \Bigr)^{1/2}.$$
The first integral on the right hand side may be estimated by means of Lemma \ref{lemma6.3} via H\"older's inequality, and the second from (\ref{6.6}). Thus one obtains
$$N_0(B;\grm,[0,1))\ll B^{l-3+\tau}(B^{2m-4-3\tau})^{1/2}(B^{2n-2+\eps})^{1/2}\ll B^{s-6-\tau/3}.$$
The desired conclusion again follows on noting that $s=12$.
\end{proof}

\begin{lemma}\label{lemma6.8} For systems of type D, one has $N_0(B;\grM,\grm)\ll B^{6-\tau}$.
\end{lemma}

\begin{proof} In the first step, by adapting the argument of the proof of Lemma \ref{lemma4.5}, we deduce from (\ref{6.6}) and Lemma \ref{lemma6.5} that
\begin{align}
N(B;\grM\times \grm)-N_0(B;\grM,\grm)&\ll B^{m-3+\eps}(\Zet_1+\Zet_2)^{1/2}\Bigl( \int_0^1|H(\bet)|^2\d\bet\Bigr)^{1/2}\notag \\
&\ll B^{m-3+\eps}(B^{2l-4-3\tau})^{1/2}(B^{2n-2+\eps})^{1/2}.\label{6.7}
\end{align}
We next estimate $N(B;\grM\times \grm)$, observing that a consideration of the underlying Diophantine equations in combination with H\"older's inequality delivers the bound
\begin{align*}
\int_\grm &F(\alp,\bet)H(\bet)\d\bet\\
&\ll \Bigl( \sup_{\bet \in \grm}|g(d_1\bet)|\Bigr) \Bigl( \int_0^1|f(d_2\bet)|^6\d\bet\Bigr)^{1/6}\prod_{i=1}^5\Bigl( \int_0^1|f(a_i\alp+b_i\bet)|^6\d\bet\Bigr)^{1/6}.
\end{align*}
We therefore deduce from Lemma \ref{lemma6.2} together with Weyl's inequality (see (\ref{2.11}) above) that
$$\int_\grm F(\alp,\bet)H(\bet)\d\bet \ll B^{3/4+\eps}(B^{3+\eps}).$$
Consequently, in view of Lemma \ref{lemma2.3}, we derive the upper bound
\begin{align*}
N(B;\grM\times \grm)&=\int_\grM G(\alp)\int_\grm F(\alp,\bet)H(\bet)\d\bet\d\alp \\
&\ll B^{15/4+\eps}\int_\grM |G(\alp)|\d\alp \ll B^{15/4+\eps}(B^{2+\eps}).
\end{align*}
On recalling (\ref{6.7}), we conclude that
$$N_0(B;\grM,\grm)\ll B^{s-6-\tau}+B^{6-\tau}\ll B^{6-\tau},$$
and the proof of the lemma is complete.
\end{proof}

The treatment of systems of type D is now completed just as in the analogous argument for systems of type B in \S4. By combining the conclusions of Lemmata \ref{lemma6.4}, \ref{lemma6.6}, \ref{lemma6.7} and \ref{lemma6.8}, we confirm by means of Lemma \ref{lemma2.1} that
$$N(B)\ge N(B;\grN)+O(B^6\calL^{-1})\gg B^6.$$
This completes the proof of Theorem \ref{theorem1.1} for systems of type D, subject to the validity of HRH.

\section{The anticipated asymptotic formula: preliminaries} In this section, we turn to the proof that $N(B)$ is asymptotically at least as large as anticipated. Consider a system of the shape (\ref{1.1}) subject to the hypotheses of the statement of Theorem \ref{theorem1.2}. The conclusion of this theorem is already supplied by \cite[Theorem 1.1]{BW2011} when $s\ge 14$, so there is no loss of generality in restricting to the situations with $s=12$ and $13$. A moment of thought reveals that the triple $(l,m,n)$ associated with the system (\ref{1.1}) must take one of three shapes, namely:
\begin{enumerate}
\item[(E)] $(4,4,4)$, $(4,5,4)$ or $(5,4,4)$,
\item[(F)] $(3,5,5)$,
\item[(G)] $(5,5,3)$.
\end{enumerate}
We now modify the argument of \S\S2--6 in order to accommodate the modest changes involved in obtaining an asymptotic formula for $N(B)$. We take the expedient approach of adopting all notation from those sections without further comment, unless noted otherwise, and thereby economise on space.\par

We begin by modifying the definitions of the generating functions defined in (\ref{2.2}) and (\ref{2.3}) by putting
$$F(\alp,\bet)=\prod_{i=1}^lf(a_i\alp+b_i\bet),\quad G(\alp)=\prod_{j=1}^mf(c_j\alp)\quad \text{and}\quad H(\bet)=\prod_{k=1}^nf(d_k\bet).$$
With the definition (\ref{2.4}) unchanged, one finds by orthogonality that
\begin{equation}\label{7.1}
N(B)=N(B;[0,1)^2).
\end{equation}

\begin{lemma}\label{lemma7.1} For systems of type E, F and G, one has
$$N(B;\grN) =\calC B^{s-6}+O(B^{s-6}\calL^{-1}).$$
\end{lemma}

\begin{proof} We may apply the argument of the proof of Lemma \ref{lemma2.1}. The presence of additional classical Weyl sums, rather than their smooth brethren, ensures that the analysis underlying the proof of the latter lemma not only remains valid, but proceeds in a manner more pedestrian than in \S2 (see also the proof of \cite[Lemma 3.1]{BW2011}). Thus one obtains the asymptotic formula
$$N(B;\grN)=\grS\grJ(B)+O(B^{s-6}\calL^{-1}),$$
where $\grS=\sum\limits_{q=1}^\infty A(q)$, with $A(q)$ defined as in (\ref{2.7}), and
$$\grJ(B)=\iint_{\dbR^2} \prod_{i=1}^lw(a_i\xi+b_i\zet)\prod_{j=1}^mw(c_j\xi)\prod_{k=1}^n w(d_k\zet)\d\xi \d\zet .$$
A conventional analysis akin to that described in \S2 reveals that, provided the system (\ref{1.1}) admits non-singular $p$-adic solutions for each prime number $p$, then $1\ll \grS\ll 1$. Moreover, for a suitable positive constant $\grJ$, one finds that $\grJ(B)=\grJ B^{s-6}$. Here, in the notation introduced in the preamble to the statement of Theorem \ref{theorem1.2}, one has $\grJ=v_\infty$ and $\grS=\prod_pv_p$ (compare the treatment of \cite{BW2011}). This confirms the asymptotic formula claimed in the statement of the lemma, with $\calC=\grS \grJ$.
\end{proof}

The exponential sum $g(\tet)$, avoiding as it does summands with arguments close to $0$, has slightly better behaviour on major arcs than does $f(\tet)$. We therefore record surrogates for Lemmata \ref{lemma2.2} and \ref{lemma2.3} of use in later sections.

\begin{lemma}\label{lemma7.2} Let $a$ be a fixed non-zero integer, and let $b$ be a non-zero rational number. Then when $\del >0$, one has
$$\sup_{\lam \in \dbR}\int_\grM |f(a\tet)^{3+\del}f(b\tet+\lam)^2|\d\tet \ll B^{2+\del}$$
and
$$\int_\grM |f(a\tet)|^{4+\del}\d\tet \ll B^{1+\del}.$$
\end{lemma}

\begin{proof} The first estimate follows via the argument of the proof of \cite[Lemma 9]{BW2007a}, and the second from the methods of \cite[\S\S4.3 and 4.4]{Vau1997}.
\end{proof}

\begin{lemma}\label{lemma7.3} Suppose that the integer $w$ is non-zero. Then for $m\ge 4$, one has
$$R_1(w;\grM)\ll B^{m-3}L^\eps\quad \text{and}\quad R_1(0;\grM)\ll B^{m-3+\eps}.$$
Similarly, when $n\ge 4$, one has
$$R_2(w;\grM)\ll B^{n-3}L^\eps\quad \text{and}\quad R_2(0;\grM)\ll B^{n-3+\eps}.$$
Finally, when $m\ge 5$ one has $R_1(w;\grM)\ll B^{m-3}$ for all integers $w$.
\end{lemma}

\begin{proof} On recalling the definitions of $R_1(w;\grM)$ and $R_2(w;\grM)$ from the preamble to Lemma \ref{lemma3.3}, these estimates follow from the methods of \cite[\S\S4.3 and 4.4]{Vau1997}. We note that a precise form of these upper bounds may be derived from \cite[equations (1.3) and (1.4)]{Kaw1996}.
\end{proof}

We finish by recording some mean value estimates for Weyl sums. In this context, when $r\ge t\ge 3$ and $\lam_1,\ldots,\lam_r$ are fixed non-zero integers, we write
$$D_t(\tet)=\prod_{i=1}^tf(\lam_i\tet).$$

\begin{lemma}\label{lemma7.4} One has
\begin{align*}
\int_0^1|D_t(\tet)|^2\d\tet &\ll_\bflam B^{2t-5/2}L^{\eps-3/2}\quad (3\le t\le r),\\
\int_\grm |D_t(\tet)|^2\d\tet &\ll_\bflam B^{2t-3}L^{\eps-3}\quad (4\le t\le r),\\
\int_\grm |D_t(\tet)|^2\d\tet &\ll_\bflam B^{2t-7/2}L^{\eps-5/2}\quad (5\le t\le r).
\end{align*}
\end{lemma}

\begin{proof} The upper bounds
\begin{equation}\label{7.2}
\int_0^1|f(\alp)|^4\d\alp \ll B^2\quad \text{and}\quad \int_\grm |f(\alp)|^8\d\alp \ll B^5L^{\eps-3},
\end{equation}
that follow, respectively, from Hooley \cite[Theorem 1]{Hoo1980} and Boklan \cite{Bok1993}, combine through the medium of Schwarz's inequality to give
$$\int_\grm |f(\alp)|^6\d\alp \le \Bigl( \int_0^1|f(\alp)|^4\d\alp \Bigr)^{1/2}\Bigl( \int_\grm |f(\alp)|^8\d\alp \Bigr)^{1/2}\ll B^{7/2}L^{\eps-3/2}.$$
In view of Lemma \ref{lemma7.2}, therefore, one finds that
$$\int_0^1|f(\alp)|^6\d\alp =\int_\grM|f(\alp)|^6\d\alp +\int_\grm|f(\alp)|^6\d\alp \ll B^3+B^{7/2}L^{\eps -3/2}.$$
By applying H\"older's inequality, when $3\le t\le r$ one obtains 
$$\int_0^1|D_t(\tet)|^2\d\tet \le g(0)^{2t-6}\int_0^1|f(\tet)|^6\d\tet \ll B^{2t-6}(B^{7/2}L^{\eps-3/2}).$$
This establishes the first estimate of the lemma.\par

The second estimate of the lemma follows from the second bound of (\ref{7.2}), following an application of H\"older's inequality in a manner similar to that above. For the third estimate of the lemma, we begin by recalling the sharpened version of Weyl's inequality
$$\sup_{\alp\in \grm}|f(\alp)|\ll B^{3/4}L^{1/4+\eps},$$
available from \cite{Vau1986} via the work of Hall and Tenenbaum \cite{HT1988}. This leads from the second estimate of (\ref{7.2}) to the upper bound
$$\int_\grm|f(\alp)|^{10}\d\alp \ll \Bigl( \sup_{\alp\in\grm}|f(\alp)|\Bigr)^2\int_\grm |f(\alp)|^8\d\alp \ll (B^{3/4}L^{1/4+\eps})^2B^5L^{\eps-3}.$$
The final estimate of the lemma therefore follows once again by employing H\"older's inequality.
\end{proof}

\section{Systems of type E} The treatment of systems of type E is similar to that of systems of type B in \S4, and we imitate the latter throughout this section.

\begin{lemma}\label{lemma8.1} For systems of type E, one has $N(B;\grK)\ll B^{s-6}\calL^{-1}$.
\end{lemma}

\begin{proof} Define the mean values
\begin{align*}
U_{ijk}&=\iint_\grK |f(c_j\alp)|^{10/3}|f(d_k\bet)|^{13/3}|f(a_i\alp+b_i\bet)|^2\d\alp\d\bet,\\
V_{ijk}&=\iint_\grK |f(c_j\alp)|^{13/3}|f(d_k\bet)|^{10/3}|f(a_i\alp+b_i\bet)|^2\d\alp\d\bet
\end{align*}
and put
$$\Psi(\alp,\bet)=\prod_{i=1}^l|f(a_i\alp+b_i\bet)|^{24(l-2)m}\prod_{j=1}^m|f(c_j\alp)|^{4l(6m-23)}\prod_{k=1}^4|f(d_k\bet)|^{lm}.$$
Then an application of H\"older's inequality reveals that
\begin{equation}\label{8.1}
N(B;\grK)\le \Bigl( \sup_{(\alp,\bet)\in \grn}\Psi(\alp,\bet)\Bigr)^{1/(24lm)}\prod_{i=1}^l\prod_{j=1}^m\prod_{k=1}^4(U_{ijk}V_{ijk})^{1/(8lm)}.
\end{equation}

\par The argument of the proof of \cite[Lemma 10]{BW2007a} shows that
$$\sup_{(\alp,\bet)\in \grn}\Psi(\alp,\bet)\ll (B^{l+m-17/3})^{24lm}Q^{-1/10}.$$
As a consequence of Lemma \ref{lemma7.2}, meanwhile, one has
\begin{align*}
U_{ijk}&\le \Bigl( \int_\grM|f(d_k\bet)|^{13/3}\d\bet \Bigr)\Bigl( \sup_{\lam\in\dbR}\int_\grM |f(c_j\alp)|^{10/3}|f(a_i\alp+\lam)|^2\d\alp \Bigr)\\
&\ll (B^{4/3})(B^{7/3})=B^{11/3},
\end{align*}
and a symmetric argument yields the estimate $V_{ijk}\ll B^{11/3}$. Combining these estimates within the framework of (\ref{8.1}), we conclude that
$$N(B;\grK)\ll (B^{l+m-17/3}Q^{-1/(240lm)})(B^{11/3})\ll B^{s-6}Q^{-1/4800},$$
and the bound asserted in the lemma follows immediately.
\end{proof}

We proceed now as in \S4, adopting the notation introduced in the discussion associated with Lemmata \ref{lemma3.2} and \ref{lemma3.3}. For systems of type E, the truncation parameter is fixed to be $T=B^{l-3}L$. As in \S4, we pause at this point to establish an estimate for the auxiliary quantity $\Zet_i$.

\begin{lemma}\label{lemma8.2} For systems of type E, one has $\Zet_i\ll B^{2l-3}L^{\eps-3}$ $(i=1,2)$.
\end{lemma}

\begin{proof} Since $l\ge 4$, the estimate
$$\int_\grM F(\alp,0)e(-\alp u)\d\alp \ll B^{l-3}L^\eps\ (u\ne 0)$$
may be established just as in the proof of Lemma \ref{lemma7.3}. Meanwhile, H\"older's inequality combines with the first estimate of (\ref{7.2}) to supply the bound
\begin{equation}\label{8.2}
\int_0^1 F(\alp,0)\d\alp \ll B^{l-2},
\end{equation}
and Lemma \ref{lemma7.4} delivers the estimate
\begin{equation}\label{8.2a}
\int_\grm |F(\alp,0)|^2\d\alp \ll B^{2l-3}L^{\eps-3}.
\end{equation}
The argument of the proof of Lemma \ref{lemma4.2} therefore demonstrates that
\begin{align*}
\sum_{(u,v)\in \grX_1}\rho(u,v)&\le B^{3-l}L^{-1}\sum_{\substack{u\in\grX\\ \rho_1(u)>B^{l-3}L}}\rho_1(u)^2\\
&\ll B^{3-l}L^{-1}\Bigl( \rho_1(0)^2+\int_\grm |F(\alp,0)|^2\d\alp \Bigr) \\
&\ll B^{3-l}L^{-1}((B^{l-2})^2+B^{2l-3}L^{\eps-3}),
\end{align*}
and hence
$$\Zet_1\le B^{l-3}L\sum_{(u,v)\in \grX_1}\rho(u,v)\ll B^{2l-3}L^{\eps-3}.$$
Similarly, and again following the argument of the proof of Lemma \ref{lemma4.2}, one finds that
$$\Zet_2\ll \Bigl( \rho_2(0)^2+\int_\grm |F(0,\bet)|^2\d\bet \Bigr)\ll (B^{l-2})^2+B^{2l-3}L^{\eps-3}.$$
The conclusion of the lemma therefore follows both for $i=1$ and $i=2$.
\end{proof}

\begin{lemma}\label{lemma8.3} For systems of type E, one has
$$N_0(B;\grM,\grM)-N(B;\grM\times \grM)\ll B^{s-6}L^{\eps-3/2}.$$
\end{lemma}

\begin{proof} Since in present circumstances one has $n\ge 4$, it follows from Lemmata \ref{lemma7.2} and \ref{lemma7.4} via H\"older's inequality that
\begin{equation}\label{8.3}
\int_0^1|H(\bet)|^2\d\bet\le \int_\grM|H(\bet)|^2\d\bet +\int_\grm |H(\bet)|^2\d\bet \ll B^{2n-3}.
\end{equation}
Following the argument of the proof of Lemma \ref{lemma4.3}, therefore, one obtains
$$\Ups_i\le \Zet_i^{1/2}\Bigl( \int_0^1|H(\bet)|^2\d\bet\Bigr)^{1/2}\ll B^{l+n-3}L^{\eps-3/2}.$$
For systems of type E one has $m\ge 4$, and so it follows from Lemma \ref{lemma7.3} that
\begin{align*}N(B;\grM\times \grM)&-N_0(B;\grM,\grM)\\
&=\sum_{(u,v)\in \grX_1\cup \grX_2}\rho(u,v)R_1(u;\grM)R_2(v;\grM)\\
&\ll B^{m-3}\Bigl( B^\eps \sum_{v\in \grX}\rho(0,v)|R_2(v;\grM)|+L^\eps (\Ups_1+\Ups_2)\Bigr).
\end{align*}
But since $l\ge 4$ and $n\ge 4$, one finds from (\ref{8.2}) and Lemma \ref{lemma7.3} that
$$\sum_{v\in \grX}\rho(0,v)=\rho_1(0)\ll B^{l-2}\quad \text{and}\quad R_2(v;\grM)\ll B^{n-3+\eps}.$$
Consequently,
\begin{align*}
N(B;\grM\times \grM)-N_0(B;\grM,\grM)&\ll B^{m-3}(B^{l+n-5+\eps}+B^{l+n-3}L^{\eps-3/2})\\
&\ll B^{s-6}L^{\eps-3/2},
\end{align*}
and the proof of the lemma is complete.
\end{proof}

\begin{lemma}\label{lemma8.4} For systems of type E, one has $N_0(B;\grm,[0,1))\ll B^{s-6}L^{\eps-1/2}$.
\end{lemma}

\begin{proof} By adapting the argument of the proof of Lemma \ref{lemma4.4} to the present context, one obtains
$$N_0(B;\grm,[0,1))\ll B^{l-3}L\Bigl(\int_\grm |G(\alp)|^2\d\alp\Bigr)^{1/2}\Bigl( \int_0^1|H(\bet)|^2\d\bet \Bigr)^{1/2}.$$
The first integral on the right hand side may be estimated by means of Lemma \ref{lemma7.4}, and the second from (\ref{8.3}). Thus one finds that
$$N_0(B;\grm,[0,1))\ll B^{l-3}L(B^{2m-3}L^{\eps-3})^{1/2}(B^{2n-3})^{1/2}\ll B^{s-6}L^{\eps-1/2}.$$
This completes the proof of the lemma.
\end{proof}

\begin{lemma}\label{lemma8.5} For systems of type E, one has $N_0(B;\grM,\grm)\ll B^{s-6}L^{\eps-1/2}$.
\end{lemma}

\begin{proof} By adapting the argument of the proof of Lemma \ref{lemma4.4} to the present context, one obtains
$$N_0(B;\grM,\grm)\ll B^{l-3}L\Bigl(\int_0^1 |G(\alp)|^2\d\alp\Bigr)^{1/2}\Bigl( \int_\grm|H(\bet)|^2\d\bet \Bigr)^{1/2}.$$
The first integral on the right hand side may be estimated by means of a variant of (\ref{8.3}), and the second by means of Lemma \ref{lemma7.4}. Thus one finds that
$$N_0(B;\grM,\grm)\ll B^{l-3}L(B^{2m-3})^{1/2}(B^{2n-3}L^{\eps-3})^{1/2}\ll B^{s-6}L^{\eps-1/2}.$$
This completes the proof of the lemma.
\end{proof}

We may now complete the proof of Theorem \ref{theorem1.2} for systems of type E. We simply combine (\ref{7.1}) with Lemmata \ref{lemma8.1}, \ref{lemma8.3}, \ref{lemma8.4} and \ref{lemma8.5} as in the analogous argument completing the discussion of \S4, obtaining
$$N(B)\ge N(B;\grN)+O(B^{s-6}\calL^{-1})=\calC B^{s-6}+O(B^{s-6}\calL^{-1}).$$

\section{Systems of type F} In common with the treatment of systems of type C in \S5, our argument for systems of type F, wherein $s=13$ and $(l,m,n)=(3,5,5)$, may be substantially abbreviated by adjusting the argument of \S8 through modification of the generating functions $F(\alp,\bet)$, $G(\alp)$ and $H(\bet)$. We begin with a discussion of the pruning operation implicit in the estimation of $N(B;\grK)$.\par

\begin{lemma}\label{lemma9.1} For systems of type F, one has $N(B;\grK)\ll B^7\calL^{-1}$.
\end{lemma}

\begin{proof} Define the mean values
$$U_j=\iint_\grK |f(c_j\alp)f(d_j\bet)|^{9/2}\d\alp\d\bet ,$$
and put
$$\Psi(\alp,\bet)=\prod_{i=1}^3|f(a_i\alp+b_i\bet)|^{10}\prod_{j=1}^5|f(c_j\alp)f(d_j\bet)|.$$
Then an application of H\"older's inequality reveals that
\begin{equation}\label{9.1}
N(B;\grK)\le \Bigl( \sup_{(\alp,\bet)\in \grn}\Psi(\alp,\bet)\Bigr)^{1/10}\prod_{j=1}^5U_j^{1/5}.
\end{equation}
The argument of the proof of \cite[Lemma 10]{BW2007a} shows that
$$\sup_{(\alp,\bet)\in \grn}\Psi(\alp,\bet)\ll B^{40}Q^{-1/10}.$$
As a consequence of Lemma \ref{lemma7.2}, meanwhile, one has
\begin{align*}
U_j&\le \Bigl( \int_\grM |f(c_j\alp)|^{9/2}\d\alp\Bigr)\Bigl( \int_\grM|f(d_j\bet)|^{9/2}\d\bet \Bigr)\\
&\ll (B^{3/2})^2=B^3.
\end{align*}
Combining these estimates with (\ref{9.1}), we conclude that
$$N(B;\grK)\ll (B^{40}Q^{-1/10})^{1/10}B^3\ll B^7Q^{-1/100}.$$
This completes the proof of the lemma.
\end{proof}

Our next step is to relabel the coefficients of the system (\ref{1.1}) so that $\mtil=m-1$, $\ntil=n-1$, $\ltil=l+2$, which is to say that $(\ltil,\mtil,\ntil)=(5,4,4)$, and to put
$$\ctil_j=c_j\quad \text{and}\quad \dtil_j=d_j\quad (1\le j\le 4),$$
and
$$(\atil_i,\btil_i)=(a_i,b_i)\quad (i=1,2,3),\quad (\atil_4,\btil_4)=(0,d_5)\quad (\atil_5,\btil_5)=(c_5,0).$$
As in the discussion of \S5, we then define the generating functions $\Ftil(\alp,\bet)$, $\Gtil(\alp)$ and $\Htil(\bet)$ as in the respective definitions of $F(\alp,\bet)$, $G(\alp)$ and $H(\bet)$ in (\ref{2.2}) and (\ref{2.3}), save that in the present context the integers $l$, $m$, $n$, and the coefficients $a_i$, $b_i$, $c_j$ and $d_k$, are to be decorated by tildes. Further notation from \S\S2 and 3 is again understood to have the meaning naturally inferred in like manner when decorated by a tilde. An examination of the argument of \S8, leading from the discussion preceding Lemma \ref{lemma8.2} to the conclusion of the section, now reveals that no adjustment is necessary in order to accommodate the change of circumstances implicit in our present analysis. Here it is worth noting that, despite the fact that we now have $\ltil=5$ and $\atil_4=0$, the presence of four non-zero coefficients in the equation (\ref{3.8}) ensures that the analogue of the upper bounds (\ref{8.2}) and (\ref{8.2a}) remain valid. Thus one obtains $\Zettil_1\ll B^7L^{\eps-3}$, and by means of a symmetric argument also $\Zettil_2\ll B^7L^{\eps-3}$. The analogue of Lemma \ref{lemma8.3} delivers the bound
$$\Ntil_0(B;\grM,\grM)-\Ntil(B;\grM\times \grM)\ll B^7L^{\eps-3/2},$$
and analogues of Lemmata \ref{lemma8.4} and \ref{lemma8.5} yield
$$\Ntil_0(B;\grm,[0,1))\ll B^7L^{\eps-1/2}\quad \text{and}\quad \Ntil_0(B;\grM,\grm)\ll B^7L^{\eps-1/2}.$$
We therefore deduce that
\begin{align*}
\Ntil(B)&\ge \Ntil_0(B;\grM,\grM)+\Ntil_0(B;\grM,\grm)+\Ntil_0(B;\grm,[0,1))\\
&=\Ntil(B;\grM\times \grM)+O(B^7L^{\eps-1/2})=N(B;\grM\times \grM)+O(B^7L^{\eps-1/2}).
\end{align*}
Finally, we conclude from Lemmata \ref{lemma7.1} and \ref{lemma9.1} that
$$N(B)=\Ntil(B)\ge N(B;\grN)+O(B^7\calL^{-1})=\calC B^7+O(B^7\calL^{-1}),$$
and this completes the proof of Theorem \ref{theorem1.2} for systems of type F.

\section{Systems of type G} Our argument when $(l,m,n)=(5,5,3)$ is motivated by the treatment of systems of type D in \S6.

\begin{lemma}\label{lemma10.1} For systems of type G, one has $N(B;\grK)\ll B^7\calL^{-1}$.
\end{lemma}

\begin{proof} Define the mean values
$$U_{ij}=\iint_\grK |f(c_i\alp)|^{9/2}|f(d_j\bet)^3f(a_i\alp+b_i\bet)^2|\d\alp\d\bet$$
and put
$$\Psi(\alp,\bet)=\prod_{i=1}^5|f(a_i\alp+b_i\bet)^6f(c_i\alp)|.$$
Then an application of H\"older's inequality reveals that
\begin{equation}\label{10.1}
N(B;\grK)\ll \Bigl( \sup_{(\alp,\bet)\in\grn}\Psi(\alp,\bet)\Bigr)^{1/10}\prod_{i=1}^5\prod_{j=1}^3U_{ij}^{1/15}.
\end{equation}
The argument of the proof of \cite[Lemma 10]{BW2007a} shows that
$$\sup_{(\alp,\bet)\in\grm}\Psi(\alp,\bet)\ll B^{35}Q^{-1/10}.$$
As a consequence of Lemma \ref{lemma7.2}, on the other hand, one has
\begin{align*}
U_{ij}&\ll \Bigl( \int_\grM |f(c_i\alp)|^{9/2}\d\alp\Bigr) \Bigl(\sup_{\lam\in\dbR}\int_\grM |f(d_j\bet)^3f(b_i\bet+\lam)^2|\d\bet\Bigr) \\
&\ll (B^{3/2})(B^2)=B^{7/2}.
\end{align*}
Combining these estimates within (\ref{10.1}), we conclude that
$$N(B;\grK)\ll (B^{35}Q^{-1/10})^{1/10}B^{7/2}\ll B^7Q^{-1/100}.$$
The conclusion of the lemma now follows.
\end{proof}

We proceed now as in \S4, adopting the notation introduced in the discussion prior to Lemmata \ref{lemma3.2} and \ref{lemma3.3}. For systems of type G, the truncation parameter is fixed to be $T=B^{l-3}L$. Our immediate goal is to derive a bound for the quantity $\Zet_i$ introduced prior to Lemma \ref{lemma4.2}.

\begin{lemma}\label{lemma10.2}
For systems of type G, one has $\Zet_i\ll B^{13/2}L^{\eps-5/2}\quad (i=1,2)$.
\end{lemma}

\begin{proof} We apply the argument of the proof of Lemma \ref{lemma8.2}, noting that since $l=5$, in this instance Lemma \ref{lemma7.4} delivers the estimates
$$\int_\grm |F(\alp,0)|^2\d\alp \ll B^{2l-7/2}L^{\eps-5/2}\quad \text{and}\quad \int_\grm |F(0,\bet)|^2\d\bet \ll B^{2l-7/2}L^{\eps-5/2}.$$
Thus we obtain
$$\sum_{(u,v)\in \grX_1}\rho(u,v)\ll B^{3-l}L^{-1}((B^{l-2})^2+B^{2l-7/2}L^{\eps-5/2}),$$
and hence
$$\Zet_1\ll B^{l-3}L\sum_{(u,v)\in \grX_1}\rho(u,v)\ll B^{2l-7/2}L^{\eps-5/2}.$$
Also, though more directly,
$$\Zet_2\ll (B^{l-2})^2+B^{2l-7/2}L^{\eps-5/2}.$$
The conclusion of the lemma now follows for $i=1$ and $2$.
\end{proof}

\begin{lemma}\label{lemma10.3} For systems of type G, one has
$$N_0(B;\grM,\grM)-N(B;\grM\times \grM)\ll B^7L^{-1}.$$
\end{lemma}

\begin{proof} In the present situation one has $n=3$, and so Lemma \ref{lemma7.4} yields
\begin{equation}\label{10.2}
\int_0^1|H(\bet)|^2\d\bet \ll B^{7/2}L^{\eps-3/2}.
\end{equation}
Following the argument of the proof of Lemma \ref{lemma4.3}, one obtains
$$\Ups_i\le \Zet_i^{1/2}\Bigl( \int_0^1|H(\bet)|^2\d\bet\Bigr)^{1/2}\ll B^5L^{\eps-2}.$$
But for systems of type G one has $m=5$, and so it follows from Lemma \ref{lemma7.3} that $R_1(u;\grM)\ll B^2$ uniformly in $u$. We therefore conclude that
\begin{align*}
N(B;\grM\times \grM)-N_0(B;\grM,\grM)&\ll \sum_{(u,v)\in \grX_1\cup \grX_2}\rho(u,v)R_1(u;\grM)R_2(v;\grM)\\
&\ll B^2(\Ups_1+\Ups_2)\ll B^7L^{\eps-2},
\end{align*}
and the proof of the lemma is complete.
\end{proof}

\begin{lemma}\label{lemma10.4} For systems of type G, one has $N_0(B;\grm,[0,1))\ll B^7L^{\eps-1}$.
\end{lemma}

\begin{proof} Adapting the argument of the proof of Lemma \ref{lemma4.4} to the present situation, one finds that
$$N_0(B;\grm,[0,1))\ll B^{l-3}L\Bigl( \int_\grm |G(\alp)|^2\d\alp \Bigr)^{1/2}\Bigl( \int_0^1|H(\bet)|^2\d\bet \Bigr)^{1/2}.$$
The first integral on the right hand side may be estimated via Lemma \ref{lemma7.4}, and the second by means of (\ref{10.2}). Thus one obtains
$$N_0(B;\grm,[0,1))\ll B^2L(B^{13/2}L^{\eps-5/2})^{1/2}(B^{7/2}L^{\eps-3/2})^{1/2}\ll B^7L^{\eps-1}.$$
This completes the proof of the lemma.
\end{proof}

\begin{lemma}\label{lemma10.5} For systems of type G, one has $N_0(B;\grM,\grm)\ll B^7L^{-1}$.
\end{lemma}

\begin{proof} First, adapting the argument of the proof of Lemma \ref{lemma4.5}, we infer from Lemma \ref{lemma7.3} that
\begin{align*}N(B;\grM\times \grm)-N_0(B;\grM,\grm)&=\sum_{(u,v)\in \grX_1\cup \grX_2}\rho(u,v)R_1(u;\grM)R_2(v;\grm)\\
&\ll B^2(\Zet_1+\Zet_2)^{1/2}\Bigl( \int_0^1|H(\bet)|^2\d\bet\Bigr)^{1/2}.
\end{align*}
Consequently, from (\ref{10.2}) and Lemma \ref{lemma10.2}, one obtains
\begin{align}
N(B;\grM\times \grm)-N_0(B;\grM,\grm)&\ll B^2(B^{13/2}L^{\eps-5/2})^{1/2}(B^{7/2}L^{\eps-3/2})^{1/2}\notag \\
&=B^7L^{\eps-2}.\label{10.3}
\end{align}

We next estimate $N(B;\grM\times \grm)$, observing that an application of H\"older's inequality together with (\ref{7.2}) and \cite[Theorem 2]{Vau1986} yields
\begin{align*}
\int_\grm F(\alp,\bet)H(\bet)\d\bet &\ll \prod_{k=1}^3\Bigl( \int_\grm |f(d_k\bet)|^8\d\bet \Bigr)^{1/8}\prod_{i=1}^5\Bigl( \int_0^1|f(a_i\alp+b_i\bet)|^8\d\bet\Bigr)^{1/8}\\
&\ll (B^5L^{\eps-3})^{3/8}(B^5)^{5/8}\ll B^5L^{\eps-9/8}.
\end{align*}
Thus, one deduces from Lemma \ref{lemma7.3} that
$$N(B;\grM\times \grm)\ll B^5L^{\eps-9/8}\int_\grM |G(\alp)|\d\alp \ll B^7L^{-1}.$$
The conclusion of the lemma now follows by reference to (\ref{10.3}).
\end{proof}

The proof of Theorem \ref{theorem1.2} for systems of type G follows by combining (\ref{7.1}) with Lemmata \ref{lemma10.1}, \ref{lemma10.3}, \ref{lemma10.4} and \ref{lemma10.5}, just as in the analogous argument completing the analysis of \S4, and so we arrive at the lower bound
$$N(B)\ge N(B;\grN)+O(B^7\calL^{-1})=\calC B^7+O(B^7\calL^{-1}).$$

\section{Further applications} The key feature of the systems amenable to our methods is a block structure. Our methods make possible the analysis of Diophantine systems of the shape
$$\left.\begin{aligned}
&\phi(x_1,\ldots ,x_l)+\psi(y_1,\ldots,y_m)&&=0,\\
&\chi(x_1,\ldots ,x_l)&+\ome(z_1,\ldots ,z_n)&=0,
\end{aligned}\, \right\}$$
for homogeneous polynomials $\phi$, $\psi$, $\chi$, $\ome$ of degree $d$, provided that $l$, $m$, $n$ are suitably large. The simplest situations to describe are those wherein one has non-trivial minor arc estimates in mean square for each of the polynomials $\phi$, $\psi$, $\chi$, $\ome$. Such is the case, for example, when these polynomials are suitably non-singular forms in a number of variables exceeding $(d-1)2^{d-1}$, as a consequence of the work of Birch \cite{Bir1961}, and also when these polynomials are diagonal forms of degree $d$ in $d^2$ variables (see \cite{Woo2011a, Woo2011b}). In the latter case, moreover, if one restricts the variables to be smooth then one can reduce the number of variables required to $\frac{1}{2}d(\log d+\log \log d+O(1))$ (see the methods of \cite{Woo1992, Woo1995a}).\par

It may be worthwhile to be more specific concerning the diagonal examples alluded to above. Consider then the Diophantine system
\begin{equation}\label{11.1}
\left.\begin{aligned}
&a_1x_1^d+\ldots +a_lx_l^d+c_1y_1^d+\ldots +c_my_m^d&&=0,\\
&b_1x_1^d+\ldots +b_lx_l^d&+d_1z_1^d+\ldots +d_nz_n^d&=0,
\end{aligned}\, \right\}
\end{equation}
wherein $l$, $m$, $n$ are each at least $\frac{1}{2}d(\log d+\log \log d+O(1))$. Also, let $N(B)$ denote the number of integral solutions of (\ref{11.1}) with $|x_i|,|y_i|,|z_i|\le B$. Provided that the system (\ref{11.1}) admits non-singular real and $p$-adic solutions for each prime number $p$, then one may prove via our methods that $N(B)\gg B^{s-2d}$, where $s=l+m+n$. Such systems, then, are accessible to our methods when $s\ge (\frac{3}{2}+o(1))d\log d$, previous approaches being applicable only for $s\ge (2+o(1))d\log d$. With $\calC$ defined to be the product of local densities associated with the system (\ref{11.1}), on the other hand, one may obtain the lower bound $N(B)\ge (\calC+o(1))B^{s-2d}$ whenever $s\ge 3d^2$. Hitherto, such a conclusion would be available only for $s\ge 4d^2$ or thereabouts.\par

We finish by noting that at the cost of additional complications our methods may be generalised so as to be applicable to systems of three or more equations. Thus, a system of $r$ equations partitioned appropriately into $r+1$ blocks may be successfully analysed by recourse to higher moment estimates along the lines contained in our previous work \cite{BW2010a}. The conditions that must be imposed on the number of variables comprising each block become progressively more complicated to analyse as $r$ increases. When the number of blocks exceeds $r+1$, on the other hand, although inspiration may be drawn from the investigations of this paper, it seems fair to comment that the situation remains highly experimental.

\bibliographystyle{amsbracket}

\begin{thebibliography}{18}

\bibitem{BB1988}
R. C. Baker and J. Br\"udern, \emph{On pairs of additive cubic equations}, J. Reine Angew. Math. \textbf{391} (1988), 157--180.

\bibitem{Bir1961}
B. J. Birch, \emph{Forms in many variables}, Proc. Roy. Soc. Ser. A \textbf{265} (1961/1962), 245--263.

\bibitem{Bok1993}
K. D. Boklan, \emph{A reduction technique in Waring's problem, I}, Acta Arith. \textbf{65} (1993), 147--161.

\bibitem{Bru1990}
J. Br\"udern, \emph{On pairs of diagonal cubic forms}, Proc. London Math. Soc. (3) \textbf{61} (1990), 273--343.

\bibitem{Bru2009}
J. Br\"udern, \emph{Binary additive problems and the circle method, multiplicative sequences and convergent sieves}, Analytic number theory: Essays in honour of Klaus Roth, pp. 91--132, Cambridge Univ. Press, Cambridge, 2009. 

\bibitem{ARTS1}
J. Br\"udern, K. Kawada and T. D. Wooley, \emph{Additive representation in thin sequences, I: Waring's problem for cubes}, Ann. Sci. \'Ecole Norm. Sup. (4) \textbf{34} (2001), 471--501.

\bibitem{BW2007a}
J. Br\"udern and T. D. Wooley, \emph{The Hasse principle for pairs of diagonal cubic forms}, Ann. of Math. (2) \textbf{166} (2007), 865--895.

\bibitem{BW2007b}
J. Br\"udern and T. D. Wooley, \emph{The density of integral solutions for pairs of diagonal cubic equations}, Analytic Number Theory: A Tribute to Gauss and Dirichlet, Proc. of the Gauss-Dirichlet Conference (G\"ottingen, 2005) (William Duke and Yuri Tschinkel, eds.), Clay Math. Proc. \textbf{7} (2007), 57--76.

\bibitem{BW2010a} J. Br\"udern and T. D. Wooley, \emph{The asymptotic formulae in Waring's problem for cubes}, J. Reine Angew. Math. \textbf{647} (2010), 1--23.

\bibitem{BW2010}
J. Br\"udern and T. D. Wooley, \emph{On Waring's problem: three cubes and a minicube}, Nagoya Math. J. \textbf{200} (2010), 59--91.

\bibitem{BW2011}
J. Br\"udern and T. D. Wooley, \emph{Asymptotic formulae for pairs of diagonal cubic equations}, Canad. J. Math. \textbf{63} (2011), 38--54.

\bibitem{BW2012}
J. Br\"udern and T. D. Wooley, \emph{A problem of Diophantine approximation related to senary cubic forms}, in preparation.

\bibitem{Coo1972}
R. J. Cook, \emph{Pairs of additive equations}, Michigan Math. J. \textbf{19} (1972), 325--331.

\bibitem{DL1966}
H. Davenport and D. J. Lewis, \emph{Cubic equations of additive type}, Philos. Trans. Roy. Soc. London Ser. A \textbf{261} (1966), 97--136.

\bibitem{Est1962}
T. Estermann, \emph{A new application of the Hardy-Littlewood-Kloosterman method}, Proc. London Math. Soc. (3) \textbf{12} (1962), 425--444.

\bibitem{GT2010}
B. Green and T. Tao, \emph{Linear equations in primes}, Ann. of Math. (2) \textbf{171} (2010), 1753--1850.

\bibitem{HT1988}
R. Hall and G. Tenenbaum, \emph{Divisors}, Cambridge University Press, Cambridge, 1988.

\bibitem{HB1996}
D. R. Heath-Brown, \emph{A new form of the circle method, and its application to quadratic forms}, J. Reine Angew. Math. \textbf{481} (1996), 149--206.

\bibitem{HB1998}
D. R. Heath-Brown, \emph{The circle method and diagonal cubic forms}, Phil. Trans. Roy. Soc. London Ser. A \textbf{356} (1998), 673--699.

\bibitem{Hoo1980}
C. Hooley, \emph{On the numbers that are representable as the sum of two cubes}, J. Reine Angew. Math. \textbf{314} (1980), 146--173.

\bibitem{Hoo1986}
C. Hooley, \emph{On Waring's problem}, Acta Math. \textbf{157} (1986), 49--97.

\bibitem{Hoo1996}
C. Hooley, \emph{On hypothesis $K^*$ in Waring's problem}, Sieve methods, exponential sums, and their applications in number theory, pp. 175--185, Cambridge University Press, 1996.

\bibitem{Kaw1996}
K. Kawada, \emph{On the sum of four cubes}, Mathematika \textbf{43} (1996), 323--348.

\bibitem{Klo1927}
H. D. Kloosterman, \emph{On the representation of numbers in the form $ax^2+by^2+cz^2+dt^2$}, Acta Math. \textbf{49} (1927), 407--464.

\bibitem{Ser1986}
J.-P. Serre, \emph{Facteurs locaux des fonctions z\^eta des vari\'et\'es alg\'ebriques (d\'efinitions et conjectures)}, Collected Papers, vol. II, pp. 581--592, Berlin, Springer, 1986.

\bibitem{Vau1977}
R. C. Vaughan, \emph{On pairs of additive cubic equations}, Proc. London Math. Soc. (3) \textbf{34} (1977), 354--364.

\bibitem{Vau1986}
R. C. Vaughan, \emph{On Waring's problem for cubes}, J. Reine Angew. Math. \textbf{365} (1986), 122--170.

\bibitem{Vau1997}
R. C. Vaughan, \emph{The Hardy-Littlewood method}, 2nd edition, Cambridge University Press, Cambridge, 1997.

\bibitem{Woo1992}
T. D. Wooley, \emph{Large improvements in Waring's problem}, Ann. of Math. \textbf{135} (1992), 131--164.

\bibitem{Woo1995a}
T. D. Wooley, \emph{New estimates for smooth Weyl sums}, J. London Math. Soc. (2) \textbf{51} (1995), 1--13.

\bibitem{Woo1995b}
T. D. Wooley, \emph{Breaking classical convexity in Waring's problem: sums of cubes and quasi-diagonal behaviour}, Invent. Math. \textbf{122} (1995), 421--451. 

\bibitem{Woo2000}
T. D. Wooley, \emph{Sums of three cubes}, Mathematika \textbf{47} (2000), 53--61.

\bibitem{Woo2011a}
T. D. Wooley, \emph{The asymptotic formula in Waring's problem}, Internat. Math. Res. Notices (2012), No. 7, 1485--1504.

\bibitem{Woo2011b}
T. D. Wooley, \emph{Vinogradov's mean value theorem via efficient congruencing}, Ann. of Math. \textbf{175} (2012), 1575--1627.

\end{thebibliography}
\providecommand{\bysame}{\leavevmode\hbox to3em{\hrulefill}\thinspace}

\end{document}